\documentclass[12pt]{amsart}

\usepackage{amssymb,enumerate}
\usepackage{graphicx}

\usepackage{hyperref}
\hypersetup{breaklinks=true}

\numberwithin{equation}{section}
\theoremstyle{plain}
\newtheorem{theorem}[equation]{Theorem}

\newtheorem{proposition}[equation]{Proposition}
\newtheorem{lemma}[equation]{Lemma}
\newtheorem{corollary}[equation]{Corollary}

\newtheorem{assumption}[equation]{Assumption}

\theoremstyle{remark}

\newtheorem{remark}[equation]{Remark}

\theoremstyle{definition}
\newtheorem{definition}[equation]{Definition}

\newtheorem{example}[equation]{Example}

\newcommand{\lra}{\longrightarrow}
\newcommand{\ra}{\to}
\newcommand{\restr}{\mbox{\Large \(|\)\normalsize}}

\newcommand{\D}{{\mathcal D}}

\newcommand{\G}{{\mathcal G}}

\newcommand{\h}{{\mathcal H}}

\newcommand{\N}{\mathbb N}

\newcommand{\R}{\mathbb R}

\newcommand{\U}{{\mathcal U}}

\newcommand{\Z}{\mathbb Z}

\newcommand{\slice}{\operatorname{Slice}}

\newcommand{\child}{\operatorname{Ch}}

\newcommand{\const}{\operatorname{const}}

\newcommand{\desc}{\operatorname{Desc}}
\newcommand{\dist}{\operatorname{dist}}

\newcommand{\id}{\operatorname{id}}

\newcommand{\length}{\operatorname{length}}

\newcommand{\nerve}{\operatorname{Nerve}}
\newcommand{\on}{\:\mbox{\rule{0.1ex}{1.2ex}\rule{1.9ex}{0.1ex}}\:}
\newcommand{\diam}{\operatorname{diam}}

\newcommand{\pslice}{\operatorname{PSlice}}

\newcommand{\side}{\operatorname{Side}}

\newcommand{\Star}{\operatorname{St}}

\newcommand{\tstar}{\operatorname{TSt}}
\newcommand{\tStar}{\operatorname{TSt}}

\def\D{\partial}
\newcommand{\al}{\alpha}
\def\de{\delta}
\def\De{\Delta}

\def\eps{\epsilon}
\def\ga{\gamma}
\def\Ga{\Gamma}

\def\lra{\longrightarrow}

\def\om{\omega}

\def\ra{\to}

\def\si{\sigma}
\def\Si{\Sigma}

\def\th{\theta}
\def\ulim{\mathop{\hbox{$\om$-lim}}}

\newcommand{\ol}{\overline}

\begin{document}

\begin{abstract}
We give sufficient conditions for a metric space to bilipschitz embed
in $L_1$.  In particular, if $X$ is a length space and there
is a Lipschitz map $u:X\ra \R$ such that for every interval $I\subset \R$,
the connected components of $u^{-1}(I)$ have diameter $\leq \const\cdot \diam(I)$,
then $X$ admits a bilipschitz embedding in $L_1$.
As a corollary, the Laakso examples \cite{laakso}  bilipschitz embed in $L_1$, 
though they do not embed in any 
any Banach space with the Radon-Nikodym property (e.g. the space $\ell_1$
of summable sequences).

The spaces appearing the statement of the bilipschitz embedding theorem
have an alternate characterization as inverse limits of systems of metric graphs
satisfying certain additional conditions. This representation, which may be of
independent interest, is the initial part of the proof
of the bilipschitz embedding theorem.  The rest of the proof  uses the combinatorial
structure of the inverse system of graphs and a diffusion 
construction, to produce the embedding in $L_1$.

\end{abstract}

\title[Inverse limits and bilipschitz embeddings in $L_1$]{Realization of metric
spaces as inverse limits, and 
bilipschitz embedding in $L_1$}
 
\date{\today}
\author{Jeff Cheeger and Bruce Kleiner}
\thanks{J.C. supported by NSF grants DMS-0704404 and DMS-1005552.
B.K. supported by NSF grants DMS-1007508 and DMS-1105656.}

\maketitle

{\small\tableofcontents}

\section{Introduction}
\label{intro}

\subsection*{Overview}
This paper is part of a series
 \cite{crannouncement,GFDA,ckbv,ckdppi,ckmetmon,CKN09,ckmetdiff} 
which examines the relations between differentiability properties
and bilipschitz embeddability in Banach spaces.
We give a new criterion for metric spaces to bilipschitz embed in
$L_1$.    This applies to  several known families of spaces,  
illustrating the sharpness of earlier nonembedding theorems.
In the first part of the proof, we characterize a certain class of metric spaces 
as inverse limits; this may be of independent interest.   

\subsection*{Metric spaces sitting over $\R$}
We begin with a special case of 
our main embedding theorem.
\begin{theorem}
\label{thm-length_space_embedding}
Let $X$ be a length space.   Suppose  $u:X\ra \R$ is a Lipschitz map, and  there 
is a $C\in (0,\infty)$ such that for every interval $I\subset \R$, each connected component
of $u^{-1}(I)$ has diameter at most $C\cdot \diam(I)$.  Then $X$ admits a bilipschitz
embedding $f:X\ra L_1(Z,\mu)$, for some measure space $(Z,\mu)$.  
\end{theorem}

We illustrate Theorem \ref{thm-length_space_embedding}
with two simple examples:

\begin{example}[Lang-Plaut \cite{LP01}, cf. Laakso \cite{laakso}]
\label{ex-laaksodiamond}
We construct a sequence of graphs $\{X_i\}_{i\geq 0}$ where $X_i$ has
a path metric so that every edge has length $4^{-i}$.
Let $X_0$ be the unit interval $[0,1]$.  
For $i>0$, inductively construct a 
$X_i$ from $X_{i-1}$ by replacing
each  edge of $X_{i-1}$ with a copy of the  graph $\Ga$ in 
Figure \ref{fig-laaksodiamond}, rescaled by the factor $m^{-(i-1)}$.
The graphs $X_1$, $X_2$, and $X_3$ are shown.   
The sequence $\{X_i\}$ naturally forms an inverse system,
$$
X_0\stackrel{\pi_0}{\longleftarrow}\cdots \stackrel{\pi_{i}}{\longleftarrow}X_{i}
\stackrel{\pi_{i+1}}{\longleftarrow}
\cdots\, ,
$$
where the projection map
$\pi_{i-1}:X_i\ra X_{i-1}$ collapses  the copies of $\Ga$ to intervals.
The inverse limit $X_\infty$ has a metric $d_\infty$ given by 
\begin{equation}
\label{eqn-d_infty_def}
d_\infty(x,x')=\lim_{i\ra\infty}\;d_{X_i}(\pi_i^\infty(x),\pi_i^\infty(x'))\,,
\end{equation}
where $\pi_i^\infty:X_\infty\ra X_i$ denotes the canonical projection.   (Note that the
sequence of metric spaces $\{X_i\}_{i\geq 0}$ 
Gromov-Hausdorff converges to $(X_\infty,d_\infty)$.)
It is not
hard to verify that $\pi_0^\infty:(X_\infty,d_\infty)\ra [0,1]$ satisfies the 
hypotheses of Theorem \ref{thm-length_space_embedding}.
\end{example}

\begin{figure}[htb] 

\begin{center}  
\includegraphics[scale=.8]{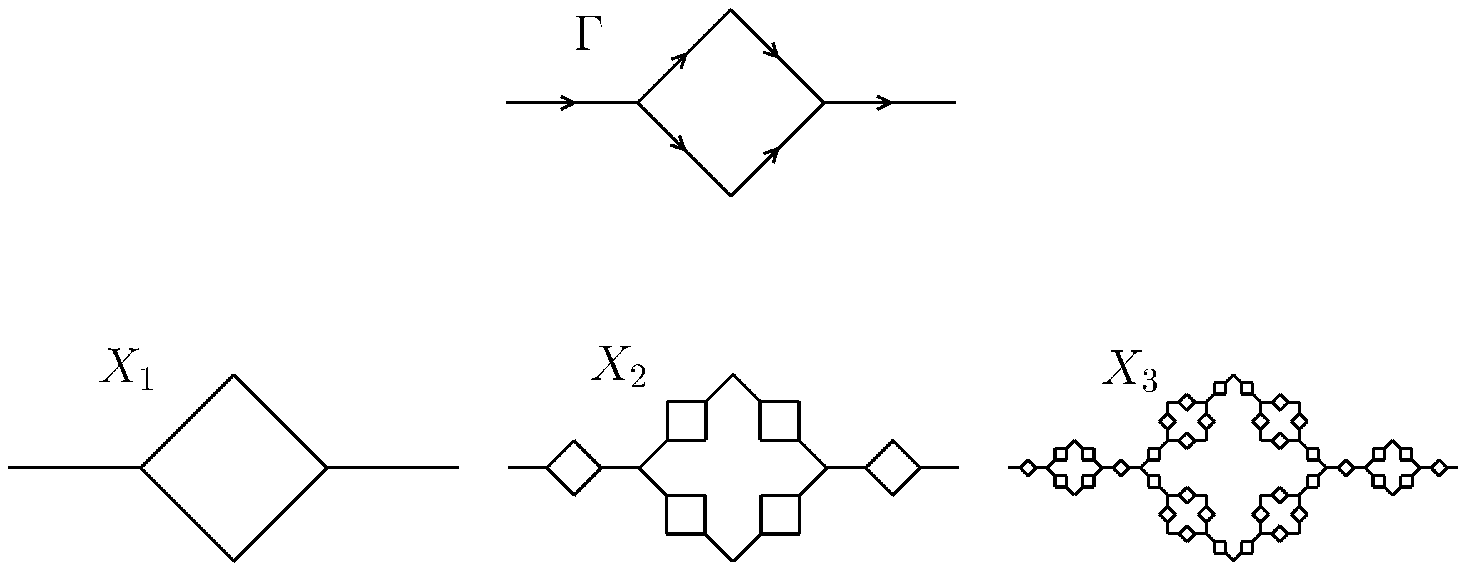} 
\caption{\label{fig-laaksodiamond}}
\end{center}

\end{figure}

\begin{example}
\label{ex-laaksoahlforsregularpi}
Construct an inverse system
$$
X_0\stackrel{\pi_0}{\longleftarrow}\cdots \stackrel{\pi_{i}}{\longleftarrow}X_{i}
\stackrel{\pi_{i+1}}{\longleftarrow}
\cdots
$$
inductively as follows.
Let $X_0=[0,1]$.  For $i>0$,   inductively 
define $X_{i-1}'$ to be the result of trisecting all edges in  
$X_{i-1}$, and let $N\subset X_{i-1}'$ be new vertices added in trisection.  Now form
$X_i$ by taking two
copies of $X_{i-1}'$ and gluing them together along $N$.   More formally,   
$$
X_i=(X_{i-1}'\times \{0,1\})/\sim\,,
$$
where $(v,0)\sim (v,1)$ for all $v\in N$.
The map  $\pi_{i-1}:X_i\ra X_{i-1}$ is induced by 
the collapsing
map $X_{i-1}'\times \{0,1\}\ni (x,j)\mapsto x\in X_{i-1}$.  Metrizing
the inverse limit $X_\infty$ as in Example \ref{ex-laaksodiamond}, the
canonical projection $X_\infty\ra X_0\simeq [0,1]\ra \R$ satisfies the assumptions
of Theorem \ref{thm-length_space_embedding}.
\end{example}

The inverse limit $X_\infty$ in Example \ref{ex-laaksoahlforsregularpi} is actually
bilipschitz homeomorphic to one
of the Ahlfors regular Laakso spaces from \cite{laakso}, see Section \ref{sec-laaksoexamples}.
Thus Theorem \ref{thm-length_space_embedding} 
implies that this Laakso space
bilipschitz embeds in $L_1$ (this special case was announced
in \cite{crannouncement}).
Laakso showed that  $X_\infty$ carries  a doubling measure 
which satisfies a Poincar\'e inequality, and using this, the nonembedding result
of \cite{ckdppi}  implies that $X_\infty$ does not bilipschitz
embed in any Banach space which satisfies the Radon-Nikodym property.  Therefore we
have:
\begin{corollary}
\label{cor-inL1notinl1}
There is a compact  Ahlfors regular (in particular doubling) 
metric measure space satisfying a Poincar\'e inequality, which bilipschitz
embeds in $L_1$, but not in any Banach space with the Radon-Nikodym property (such as
$\ell_1$).
\end{corollary}
\noindent
To our knowledge, this is the first example of a doubling space which bilipschitz
embeds in $L_1$ but not
in $\ell_1$.

We can extend Theorem \ref{thm-length_space_embedding}
by dropping the length space condition, and replacing connected components
with a metrically based variant. 

\begin{definition}
Let $Z$ be a metric space and $\de\in (0,\infty)$.  A {\bf $\de$-path} (or {\bf $\de$-chain})
in $Z$ is a finite sequence of points $z_0,\ldots,z_k\subset Z$ such that 
$d(z_{i-1},z_i)\leq \de$
for all $i\in \{1,\ldots,k\}$.  The property of belonging to a $\de$-path defines an 
equivalence relation
on $Z$, whose cosets are the {\bf $\de$-components} of $Z$.  
\end{definition}

Our main embedding result is:

\begin{theorem}
\label{thm-main}
Let $X$ be a metric space.  Suppose there is a  $1$-Lipschitz map
$u:X\ra \R$ and a constant $C\in (0,\infty)$
such that for every interval $I\subset \R$, the $\diam(I)$-components of 
$u^{-1}(I)$ have diameter at most $C\cdot \diam(I)$.
Then $X$ admits a bilipschitz
embedding $f:X\ra L_1$.  
\end{theorem}

\bigskip

\subsection*{Inverse systems of directed metric graphs, and multi-scale factorization}
Our approach to proving Theorem \ref{thm-main} is to  first show that  any map
$u:X\ra \R$ satisfying the hypothesis of theorem  
can be factored into an infinite sequence of maps, i.e.
it gives rise to a certain kind of  inverse system where $X$ reappears
(up to bilipschitz equivalence) as the inverse limit.  Strictly speaking this
result has nothing to do with embedding, and can be viewed as a kind of
multi-scale version
of monotone-light factorization (\cite{eilenberg,whyburn})
in the metric space category.

We work with a special class of  inverse systems of graphs:

\begin{definition}[Admissible inverse systems]
\label{def-admissibleinversesystem}
An inverse system indexed by the integers
$$
\cdots\stackrel{\pi_{-i-1}}{\longleftarrow}X_{-i}\stackrel{\pi_{-i}}{\longleftarrow}
\cdots\stackrel{\pi_{-1}}{\longleftarrow}
X_0\stackrel{\pi_0}{\longleftarrow}\cdots \stackrel{\pi_{i}}{\longleftarrow}X_{i}
\stackrel{\pi_{i+1}}{\longleftarrow}
\cdots\, ,
$$
is {\bf admissible} if 
for some integer $m\geq 2$
the following  conditions hold:
\begin{enumerate}
\setlength{\itemsep}{.5ex}
\item $X_i$ is a nonempty directed
graph for every $i\in \Z$. 
\item  For every $i\in \Z$, if $X_i'$ denotes the directed  graph obtained by
subdividing each edge of $X_i$ into
$m$ edges, then $\pi_i$ induces a map
$\pi_i:X_{i+1}\ra X_i'$ which is simplicial, an isomorphism on every edge,  and
direction preserving.
\item For every $i,j\in \Z$, and every $x\in X_i$, $x'\in X_j$, there is a $k\leq \min(i,j)$ 
such that $x$ and $x'$ project to the same connected component
of $X_k$.
\end{enumerate}
\end{definition}
Note that the $X_i$'s
need not be connected or have finite valence, and they may contain isolated vertices.

We endow each $X_i$ with a (generalized) path metric $d_i:X_i\times X_i\ra[0,\infty]$, where each edge is 
linearly isometric
to  the interval $[0,m^{-i}]\subset\R$.   Since we do not require the $X_i$'s to be connected, we have
$d_i(x,x')=\infty$ when $x,x'$ lie in different connected components of $X_i$.
It follows from Definition \ref{def-admissibleinversesystem}
that the projection maps $\pi_i^j:(X_j,d_j)\ra (X_i,d_i)$ are $1$-Lipschitz.

Examples \ref{ex-laaksodiamond} and \ref{ex-laaksoahlforsregularpi} provide
admissible
inverse systems in a straightforward way: for $i<0$
one simply takes $X_i$  to be a copy of $\R$ with
the standard subdivision into intervals of length $m^{-i}$,  and the projection
map $\pi_i:X_{i+1}\ra X_i$ to be the identity map.  Of course this modification does
not affect the inverse limit.

Let  $X_\infty$ be the inverse limit of the system $\{X_i\}$, and let
$\pi_i^j:X_j\ra X_i$, 
$\pi_i^\infty:X_\infty\ra X_i$ denote the canonical projections for $i\leq j\in \Z$.
We will often omit the 
superscripts and subscripts when there is no risk of confusion.

We now equip the inverse limit $X_\infty$ with a metric $\bar d_\infty$; unlike in the earlier
examples, this is not defined as a limit of pseudo-metrics $d_i\circ \pi_i^\infty$.

\begin{definition}
\label{def-dbarinfty}
Let $\bar d_\infty:X_\infty\times X_\infty\ra [0,\infty)$ be the supremal pseudo-distance 
on $X_\infty$ 
such that for every $i\in\Z$ and every vertex $v\in X_i$, if 
$$
\Star(v,X_i) =\cup\{e\mid e\;\text{is an edge of}\; X_i,\;v\in e\}
$$
is the closed star of $v$ in $X_i$, then  the inverse image of  $\Star(v,X_i)$
under the projection
map $X_\infty\ra X_i$  has diameter 
at most $2m^{-i}$.    Henceforth, unless otherwise indicated, distances in 
$X_\infty$ will refer to $\bar d_\infty$.
\end{definition}

\noindent
 In fact $\bar d_\infty$ is a metric, and for any 
distinct points $x,x'\in X_\infty$, the distance $\bar d_\infty(x,x')$ is comparable
to $m^{-i}$, where $i$ is the maximal integer such that $\{\pi_i^\infty(x),\pi_i^\infty(x')\}$
is contained in the star of some vertex $v\in X_i$;
see Section \ref{sec-notation_preliminaries}.  In 
Examples 
\ref{ex-laaksodiamond} and \ref{ex-laaksoahlforsregularpi}, the metric $\bar d_\infty$
is comparable to the metric $d_\infty$ defined using the path metrics in
(\ref{eqn-d_infty_def}); see Section \ref{sec-inverse_systems_path_metrics}.

Admissible inverse systems give rise to spaces satisfying the hypotheses of Theorem
\ref{thm-main}:
\begin{theorem}
\label{thm-easy_direction}
Let $\{X_i\}$ be an admissible inverse system.   Then there is a $1$-Lipschitz map
$\phi:X_\infty\ra\R$ which is canonical up  to post-composition with a translation,
which satisfies the assumptions of Theorem \ref{thm-main}.
\end{theorem}

The converse is also true:

\begin{theorem}
\label{thm-realization}
Let $X$ be a metric space.
Suppose $u:X\ra \R$ is a $1$-Lipschitz map, and there is a constant $C\in [1,\infty)$ such
that for every interval $I\subset \R$, the  inverse image $u^{-1}(I)\subset X$ has 
$\diam(I)$-components of diameter at most $C\cdot \diam(I)$.   Then for
any $m\geq 2$ there is an
admissible inverse system $\{X_i\}$ and a compatible system of maps $f_i:X\ra X_i$,
such that:
\begin{itemize}
\item The induced map $f_\infty:X\ra (X_\infty,\bar d_\infty)$ is $L'=L'(C,m)$-bilipschitz.
\item $u=\phi\circ f_\infty$,
where $\phi:X_\infty\lra \R$ is the $1$-Lipschitz map of Theorem
\ref{thm-easy_direction}.
\end{itemize}
\end{theorem}

Theorem \ref{thm-length_space_embedding} is a  corollary
of Theorem \ref{thm-realization} :  if 
$u:X\ra \R$ is as in Theorem \ref{thm-length_space_embedding}, then for any interval
$[a,a+r]\subset \R$, an $r$-component
of $f^{-1}([a,a+r])$ will be contained in a connected component of $f^{-1}([a-r,a+2r])$ (since
$X$ is a length space), and therefore has diameter $\leq 3C\diam(I)$.

\begin{remark}
Theorem \ref{thm-realization} implies that Examples
\ref{ex-laaksodiamond} and \ref{ex-laaksoahlforsregularpi} can be represented up to 
bilipschitz homeomorphism as 
inverse limits of many different admissible inverse systems, since the integer $m$
may be chosen freely.   
\end{remark}

\begin{remark}
Although it is not used elsewhere in the paper, in Section \ref{sec-realization_generalization} we 
prove a result  in the spirit of Theorem \ref{thm-realization} for maps $u:X\ra Y$, where 
$Y$ is a  general metric space equipped with a sequence of coverings.   
\end{remark}

\subsection*{Analogy with light mappings in the topological category}
We would like to point out that Theorems \ref{thm-easy_direction}, \ref{thm-realization} are 
analogous to certain results for topological spaces.   

Recall that a continuous map $f:X\ra Y$ is
 light  (respectively discrete, monotone)
if the point inverses $\{f^{-1}(y)\}_{y\in Y}$ are  totally disconnected (respectively discrete, connected).
If $X$ is a compact metrizable space, then $X$ has topological dimension $\leq n$
if and only if there is a light map $X\ra \R^n$; one implication comes from the
fact  that closed light maps do not decrease topological dimension \cite[Theorem 1.12.4]{engelking},  
and the other follows from a Baire category argument.

One may consider versions of light mappings in the Lipschitz  category.   One possibility is 
the notion appearing the Theorems \ref{thm-main} and \ref{thm-realization}:
\begin{definition}
A Lipschitz map $f:X\ra Y$ between metric spaces is {\bf Lipschitz light} if there is a $C\in (0,\infty)$
such that for every bounded subset $W\subset Y$, the $\diam(W)$-components of $f^{-1}(W)$
have diameter $\leq C\cdot \diam(W)$.
\end{definition}

The analog with the topological case then leads to:
\begin{definition}
A metric space $X$ has {\bf Lipschitz dimension $\leq n$} iff there is a Lipschitz light map
from $X\ra\R^n$ where $\R^n$ has the usual metric.
\end{definition}

With this definition, Theorems \ref{thm-main} and \ref{thm-realization} become results about metric spaces
of Lipschitz dimension $\leq 1$. 

To carry the topological analogy further, we note that if $f:X\ra Y$ is a light map between metric spaces
and $X$ is compact, then \cite{dyckhoff,dranishnikov}, in a variation on monotone-light factorization,
showed that there is an inverse system
$$
Y\longleftarrow X_1\longleftarrow\ldots\longleftarrow X_k\longleftarrow\ldots
$$
and a compatible family of mappings $\{g_k:X\ra X_k\}$ such that:
\begin{itemize}
\item The projections $X_k\leftarrow X_{k+1}$ are discrete.
\item  $g_k$ gives a  factorization of $f$:
$$
Y\longleftarrow X_1\longleftarrow\ldots\longleftarrow X_k\stackrel{g_k}{\longleftarrow} X\,.
$$
\item The point inverses of $g_k$ have diameter $\leq \De_k$, where $\De_k\ra 0$ as $k\ra \infty$.

\item $\{g_k\}$ induces a homeomorphism $g_\infty:X\ra X_\infty$, where $X_\infty$ is 
 the inverse limit $X_\infty$ of the system $\{X_k\}$.
\end{itemize}
Making allowances for the difference between the Lipschitz and topological categories,
this compares well with Theorem \ref{thm-realization}.

\subsection*{Embeddability and nonembeddability of  inverse limits in Banach spaces}
Theorem \ref{thm-realization} reduces the proof of Theorem \ref{thm-main} (and also
Theorem  \ref{thm-length_space_embedding})
to:

\begin{theorem}
\label{thm-bilipschitzembedding}
Let $\{X_i\}_{i\in\Z}$ be an admissible inverse system, and  $m$ be the 
parameter in Definition \ref{def-admissibleinversesystem}.
There is a constant $L=L(m)\in (0,1)$ and a $1$-Lipschitz map
$f:X_\infty\ra L_1$ such that for all $x,y\in X_\infty$, 
$$
\|f(x)-f(y)\|_{L_1}\geq L^{-1}\,\bar d_\infty(x,y)\,.
$$
\end{theorem}

In a forthcoming paper \cite{ck-piexamples}, we show that if one imposes additional
conditions on an admissible inverse system $\{X_i\}$, the  inverse limit $X_\infty$
will carry a doubling measure $\mu$
which satisfies a Poincar\'e inequality, such that for $\mu$ a.e. $x\in X_\infty$, 
the tangent space $T_xX_\infty$ (in the sense of 
\cite{cheeger}) is $1$-dimensional.   
The results apply to Examples \ref{ex-laaksodiamond} and \ref{ex-laaksoahlforsregularpi}.
Moreover, in 
these two examples -- and typically for the spaces studied in \cite{ck-piexamples} --
the Gromov-Hausdorff tangent cones at almost every point will not
be bilipschitz homeomorphic to $\R$.  The non-embedding result
of \cite{ckdppi} then implies that such spaces do not bilipschitz embed
in Banach spaces which satisfy the Radon-Nikodym property.    Combining
this with Theorem \ref{thm-bilipschitzembedding}, we therefore obtain a
large class of examples of doubling spaces which embed in $L_1$, but not in any
Banach space satisfying the Radon-Nikodym property, cf. Corollary  \ref{cor-inL1notinl1}.

\subsection*{Monotone geodesics}
Suppose $\{X_i\}_{i\in\Z}$ is an admissible inverse system,
and  $\phi:X_\infty\ra \R$ 
is as in Theorem \ref{thm-easy_direction}.  Then $\phi$
picks out a distinguished class of paths, namely the paths $\ga:I\ra X_\infty$
such that the composition $\phi\circ\ga:I\ra \R$ is a homeomorphism onto its image, i.e.
$\phi\circ\ga:I\ra \R$ is a monotone.   (This is equivalent to saying that the projection
$\pi_i\circ\ga:I\ra X_i$ is  either direction preserving or direction reversing, with respect to the
direction on $X_i$.)
It is not difficult to 
see that such a path $\ga$ is a geodesic in $X_\infty$; see Section \ref{sec-notation_preliminaries}.   
We call the image of such a path
$\ga$ a {\bf monotone geodesic segment}  (respectively {\bf monotone ray, monotone geodesic} ) if
 the image  $\phi\circ\ga(I)\subset\R$ is a segment   (respectively is a ray, is all of $\R$).  
Monotone geodesics and related structures play an important role in the proof of 
Theorem \ref{thm-bilipschitzembedding}.  In fact, the proof
of Theorem \ref{thm-bilipschitzembedding} produces an embedding
$f:X_\infty\ra L_1$  with the additional property that it
maps monotone geodesic segments in $X_\infty$  isometrically to geodesic
segments in $L_1$. 

Now suppose $u:X\ra \R$ is as  in  Theorem \ref{thm-realization}.  As above, one obtains a distinguished
family of paths $\ga:I\ra X$, those for which $u\circ\ga:I\ra\R$ is a homeomorphism onto its image.
From the assumptions on $u$, it is easy to see   that $u$ induces a bilipschitz homeomorphism
from the image $\ga(I)\subset X$ to the image $(u\circ\ga)(I)\subset\R$, so $\ga(I)$ is a bilipschitz embedded path.
We call the images of such paths {\bf monotone}, although they need not be geodesics.  If  
$f_\infty:X\ra X_\infty$ is
a homeomorphism  provided by  Theorem \ref{thm-realization}, then $f_\infty$ maps monotone paths
in $X$ to monotone segments/rays/geodesics in $X_\infty$ because 
$\phi\circ f_\infty=u$.   Therefore,
by combining Theorems \ref{thm-realization} and \ref{thm-bilipschitzembedding}, it follows that
the embedding in Theorem  \ref{thm-main} can be chosen to map monotone paths in $X$  
to geodesics in $L_1$.

\subsection*{Discussion of the proof of Theorem \ref{thm-bilipschitzembedding}}
Before entering into the construction, we recall some observations 
from \cite{ckbv-old,ckbv,ckmetmon}
which motivate the setup, and also indicate the delicacy of the
embedding problem.  

Let $\{X_i\}_{i\in\Z}$ be an admissible inverse system.

Suppose $f:X_\infty\ra L_1$ is an $L$-bilipschitz embedding, and that
$X_\infty$ satisfies a Poincar\'e inequality with respect to a doubling
measure $\mu$  (e.g. Examples \ref{ex-laaksodiamond} and \ref{ex-laaksoahlforsregularpi}).  
Then there is a version
of Kirchheim's metric diffferentiation theorem  \cite{kirchheim},
which implies that for almost every $p\in X_\infty$, if one rescales
the map $f$ and passes to a limit, one obtains an $L$-bilipschitz
embedding $f_\infty:Z\ra L_1$, where $Z$ is a Gromov-Hausdorff tangent space
of $X_\infty$, such that
$(f_\infty)\restr_{\ga}:\ga\ra L_1$ is a constant speed geodesic
for every $\ga\subset Z$ which arises as a limit of 
(a sequence of rescaled) monotone geodesics in $X_\infty$.  When
$X_\infty$ is self-similar, as in Examples 
\ref{ex-laaksodiamond} and \ref{ex-laaksoahlforsregularpi}, 
then $Z$ contains
copies of $X_\infty$, and one concludes that $X_\infty$ itself
has an $L$-bilipschitz embedding $X_\infty\ra L_1$
which restricts to a constant speed geodesic embedding on each monotone geodesic
$\ga\subset X_\infty$.  In view of this, and the fact that any
bilipschitz embedding is constrained to have this behavior 
infinitesimally, our construction has been chosen so as to automatically
satisfy the constraint, i.e. it generates maps which restrict 
to isometric embeddings on monotone geodesics.

\begin{figure}[h] 

\begin{center}  
\includegraphics[scale=.8]{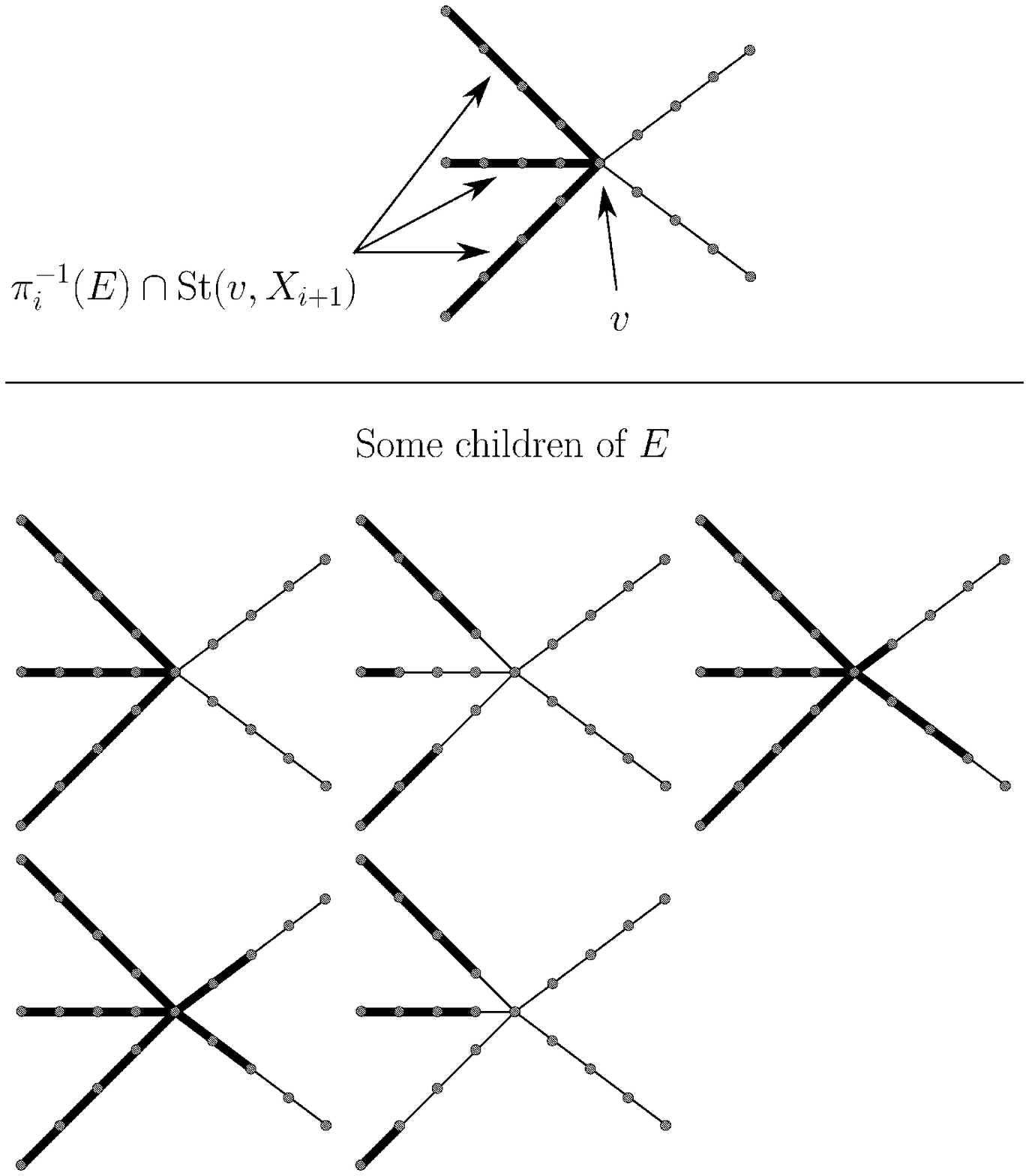} 
\caption{\label{fig-localmodification}}
\end{center}

\end{figure}

By \cite{assouad,dezalaur,ckbv}, producing a  bilipschitz embedding $f:X_\infty\ra L_1$
is equivalent to
showing that distance function $\bar d_\infty$ is comparable to a cut metric
$d_\Si$, i.e. a distance function $d_\Si$ on $X_\infty$ which is a superposition
of elementary cut metrics.  Informally speaking this means that
$$
d_\Si=\int_{2^{X_\infty}} \;d_E\,d\Si(E)\,.
$$
where $\Si$ is a cut measure on the subsets of $X_\infty$, 
and $d_E$ is the elementary cut (pseudo)metric associated with a subset $E\subset X_\infty$:
$$
d_E(x_1,x_2)=|\chi_E(x_1)-\chi_E(x_2)|\,.
$$
If $f:X_\infty\ra L_1$ restricts to an isometric embedding $f\restr_\ga:\ga\ra L_1$
for every monotone geodesic $\ga\subset X_\infty$, then one finds (informally speaking)
that the
cut measure $\Si$ is supported on  subsets $E\subset X_\infty$
with the property that for every monotone geodesic $\ga\subset X_\infty$,
the characteristic function $\chi_E$ restricts to a monotone function on $\ga$,
or equivalently, that the 
the intersections $E\cap \ga$ and $(X_\infty\setminus E)\cap \ga$ are both connected.
We call such subsets {\bf monotone}.

For simplicity we restrict the rest of our discussion to the case when $X_0\simeq \R$.
The reader may find it helpful to keep Example 
\ref{ex-laaksodiamond} in mind (modified with $X_i\simeq\R$ for $i<0$ as indicated earlier).

Motived by the above observations, the 
approach taken in the paper is to obtain the cut metric $d_\Si$ as a limit of a sequence
of cut metrics $\{d_{\Si_i'}\}_{i\geq 0}$, where $\Si_i'$ is a cut measure on $X_i$ 
supported on monotone subsets.  For technical reasons, we choose $\Si_i'$ so that
every monotone subset $E$ in the 
support of $\Si_i'$   is a subcomplex of $X_i'$
(see Definition \ref{def-admissibleinversesystem}),
and $E$ is precisely the 
set of points $x\in X_i$ such that there is a monotone
geodesic $c:[0,1]\ra X_i$ where $\pi_0^i\circ c$ is increasing, $c(0)=x$, and
$c(1)$ lies in the boundary of $E$; thus one may think of $E$ as the set of points 
``lying to the left'' of the boundary $\D E$.

We construct the sequence $\{\Si_i'\}$ inductively
as follows.  The cut measure $\Si_0'$ is the atomic measure which assigns mass
$\frac1m$ to each monotone subset of the form $(-\infty,v]$, where $v$ is vertex of
$X_0'\simeq \R$.  Inductively we construct $\Si_{i+1}'$ from $\Si_i'$ by a diffusion process.
For every monotone set $E\subset X_i$ in the support of $\Si_i'$, we take the 
$\Si_i'$-measure living on $E$, and redistribute it
over a family of monotone sets $E'\subset X_{i+1}$, called the {\bf children} of $E$.
The children of $E\subset X_i$ are monotone sets $E'\subset X_{i+1}$ obtained from
the inverse image $\pi_i^{-1}(E)$ by modifying the boundary locally: for each
vertex $v$ of $X_{i+1}$ lying in the boundary of $\pi_i^{-1}(E)$, we move the boundary
within the open star of $v$.    An example of this local modification
procedure is depicted in Figure \ref{fig-localmodification}, where $m=4$.

The remainder of the proof involves a series of estimates on the cut measures $\Si_i'$
and cut metrics $d_{\Si_i'}$, which are proved by induction on $i$ using the form of the
diffusion process, see Section \ref{sec-estimates}.
One shows that the sequence of pseudo-metrics $\{\rho_i=d_{\Si_i'}\circ \pi_i^\infty\}$
on $X_\infty$ converges geometrically to a distance function 
which will be the cut metric $d_\Si$ for a cut measure $\Si$ on $X_\infty$. 
To prove that  $d_\Si$ is  comparable
to  $d_\infty$, the  idea is to show (by induction) that the cut metric $d_{\Si_i'}$ 
resolves pairs of points $x_1,x_2\in X_i$ whose separation is $> C\,m^{-i}$.

\subsection*{Organization of the paper}
In Section \ref{sec-notation_preliminaries} we collect notation and establish some basic
properties of admissible inverse systems.   Theorem \ref{thm-easy_direction} is proved in Section 
\ref{subsec-canonical_map}.
Section \ref{sec-inverse_systems_path_metrics}
considers a special class of admissible inverse systems which come with natural metrics, e.g.
Examples \ref{ex-laaksodiamond}, \ref{ex-laaksoahlforsregularpi}.
In Section \ref{sec-realizing_metric_spaces} we prove Theorem
\ref{thm-realization}.
Sections \ref{sec-special_case}--\ref{sec-generalcase} give the proof of Theorem
\ref{thm-bilipschitzembedding}.    A special case of Theorem \ref{thm-bilipschitzembedding}
is introduced in Section \ref{sec-special_case}.   In Section  \ref{sec-slices_measures} we 
begin the proof of the special case by developing the structure of slices and associated
slice measures, which are closely
related to the monotone sets  in the above discussion of the proof of 
Theorem \ref{thm-bilipschitzembedding}.    Section \ref{sec-estimates} obtains estimates
on the slice measures which are needed for the embedding theorem.
Section \ref{sec-completion_assumption} completes the proof of
Theorem \ref{thm-bilipschitzembedding} in the special case introduced in 
Section \ref{sec-special_case}.
Section \ref{sec-generalcase} completes the proof in the general case.
Section \ref{sec-laaksoexamples} shows that the space in 
Example  \ref{ex-laaksoahlforsregularpi} is bilipschitz homeomorphic to a  Laakso space from
\cite{laakso}.
In Section \ref{sec-realization_generalization} we consider a generalization of
Theorem \ref{thm-realization} to maps $u:X\ra Y$, where $Y$ is a  general
metric space equipped with a sequence of coverings.

We refer the reader to the beginnings of the individual sections for  more detailed descriptions of
their contents.

\section{Notation and preliminaries}
\label{sec-notation_preliminaries}

In this section  $\{X_i\}_{i\in \Z}$ will be an admissible inverse system, and 
$m$ will  be the parameter appearing in Definition \ref{def-admissibleinversesystem}.

\subsection{Subdivisions, stars, and trimmed stars}
\label{subsec-subdivisions}

Let $Z$ be a graph.  Let $Z^{(k)}$ denote the $k$-fold iterated subdivision of $Z$, where
each iteration  subdivides every edge into $m$ subedges, and let $Z'=Z^{(1)}$.

If $v$ is a vertex of a graph $\G$, then $\Star(v,\G)$  and $\Star^o(v,\G)$
denote the closed and  open stars of $v$, respectively.

\begin{definition}
Let $Z$ be a graph,  
and  $v\in Z$ be a vertex.
The  {\bf trimmed star of  $v$ in $Z$} is the 
union of the edges of $Z'$ which lie in the open star $\Star^o(v,Z)$, or alternately, 
the union of the edge paths in $Z'$ starting at $v$,  with  $(m-1)$ edges.
We denote the trimmed star by $\tstar(v,Z)$.   
We will only use this when $Z=X_i$ or $Z=X_i'$ below.  
\end{definition}
Note that if $v$ is a vertex
of $X_i$, then $\tstar(v,X_i)$ is also the closed ball $\ol{B(v,\frac{m-1}{m}\cdot m^{-i})}\subset X_i$
with respect to the path metric $d_i$.

\bigskip\bigskip

\subsection{Basic properties of admissible inverse systems and the distance function
$\bar d_\infty$}\;\;
Let $\{X_i\}_{i\in \Z}$ be an admissible inverse system with inverse limit $X_\infty$.  
For every $i$, we let $V_i$ be the vertex set of $X_i$, and $V_i'$ be the vertex
set of $X_i'$.  
For all $j\geq i$, let $\pi_i^j:X_j\ra X_i$ be the composition $\pi_{j-1}\circ\ldots\circ\pi_i$.
Then $\pi_i^j:X_j\ra X_i^{(j-i)}$ is simplicial and restricts to an 
isomorphism on each edge.   It is also $1$-Lipschitz with respect to the respective path metrics
$d_j$ and $d_i$.

\begin{lemma}
\label{lem-inastar}
For every  $x,x'\in X_\infty$ there exist  $i\in \Z$, $v\in V_i$  such that
$x,x'\in (\pi_i^\infty)^{-1}(\Star(v,X_i))$.
\end{lemma}
\begin{proof}
 By (3) of Definition \ref{def-admissibleinversesystem}, there is a $j\in \Z$
such that $\pi_j(x),\pi_j(x')$ are contained in the same connected component of $X_j$.   If  $\ga\subset X_j$
is a path from $\pi_j(x)$ to $\pi_j(x')$ with $d_j$-length $N$, then for all $i\leq j$, the projection
$\pi_i^j(\ga)$ is a path in $X_i^{(j-i)}$ with $d_i$-length $\leq N$.  Therefore if $N<m^{-i}$ then 
$\pi_i^j(\ga)$ will be contained in $\Star(v,X_i)$ for some $i\in V_i$. 
\end{proof} 

Suppose $\hat d$ is a pseudo-distance on $X_\infty$ with the property that
$$
\diam_{\hat d}((\pi_j^\infty)^{-1}(\Star(v,X_j))\leq 2m^{-j}
$$ 
for all $j\in\Z$, $v\in V_j$.
Then
for every $x,x'\in X_\infty$,  we have $\hat d(x,x')\leq 2m^{-i}$, where $i\in \Z$ is as in 
Lemma \ref{lem-inastar}.
It follows that the supremum $\bar d_\infty$ of all such pseudo-distance functions  
takes finite values,
i.e. is  a well-defined pseudo-distance function.

\begin{lemma}[Alternate definition of $\bar d_\infty$]
\label{lem-altdefdbarinfty}
Suppose $x,x'\in X_\infty$.   Then $\bar d_\infty(x,x')$ is the infimum of the sums
$\sum_{k=1}^n\,2m^{-i_k}$, such that there exists a finite sequence
$$
x=x_0,\ldots,x_n=x'\in X_\infty
$$
where $\{\pi^\infty_{i_k}(x_{k-1}),\pi_{i_k}^\infty(x_k)\}$ is contained in the closed star
$\Star(v,X_{i_k})$ for 
 some  vertex $v$  of $X_{i_k}$, 
for every $k\in \{1,\ldots,n\}$.
\end{lemma}
\begin{proof}
Let $\hat d_\infty(x,x')\in[0,\infty]$ be the infimum defined above.
By Lemma \ref{lem-inastar} the infimum will be taken over a nonempty set of sequences, 
and so $\hat d_\infty(x,x')\in[0,\infty)$.  It follows that $\hat d_\infty$ 
 is a well-defined pseudo-distance
satisfying the condition that 
$\diam_{\hat d_\infty}((\pi_i^\infty)^{-1}(\Star(v,X_i))\leq 2m^{-i}$ for every 
$v\in V_i$, $i\in \Z$.
Therefore $\hat d_\infty \leq \bar d_\infty$ from the definition of $\bar d_\infty$.
On the other hand the definition of $\bar d_\infty$ and the triangle
inequality imply  $\bar d_\infty\leq \hat d_\infty$.
\end{proof}

\begin{lemma}
\label{lem-starscale}
Suppose $x,x'\in X_\infty$.
\begin{enumerate}
\setlength{\itemsep}{.5ex}
\item If $\bar d_\infty(x,x')\leq m^{-j}$ for some $j\in\Z$, then $\pi_j(x),\pi_j(x')$ belong to $\Star(v,X_j)$
for some $v\in V_j$. 
\item If $x\in X_\infty$ and 
$r\leq \frac{(m-2)}{m}m^{-j}$ for some $j\geq 0$, then $\pi_j(B(x,r))$ is contained
in the trimmed star $\tstar(v,X_j)$
for some $v\in V_j$. 
\end{enumerate}
\end{lemma}
\begin{proof}
(1).  
Pick $\eps>0$.   Since $\bar d_\infty(x,x')\leq m^{-j}$, there is  a sequence
$x=y_0,\ldots,y_k=x'\in X_\infty$, where for $\ell\in \{1,\ldots,k\}$, the points  $y_{\ell-1},y_\ell$
lie in $\pi_{i_\ell}^{-1}(\Star(v_\ell,X_{i_\ell}))$, $v_\ell\in V_{i_\ell}$, and 
$$
\sum_\ell 2m^{-i_\ell}\leq m^{-j}+\eps\,.
$$
Taking $\eps< m^{-j}$, we may assume
that $i_\ell\geq j$ for all $\ell\in \{1,\ldots,k\}$.    
Since $\pi_{i_\ell}(y_{\ell-1}),\pi_{i_\ell}(y_\ell)\in \Star(v_\ell,X_{i_\ell})$, there is a path from
$\pi_{i_\ell}(y_{\ell-1})$ to $\pi_{i_\ell}(y_\ell)$
in  $X_{i_\ell}$ of $d_{i_\ell}$-length $\leq 2m^{-i_\ell}$.  
Since $\pi_j^{i_\ell}:(X_{i_\ell},d_{i_\ell})\ra (X_j,d_j)$ is $1$-Lipschitz for all $j$, we get that there is an 
 path from $\pi_j(x)$ to $\pi_j(x')$ in $X_j$ with $d_j$-length at most 
$$
\sum_\ell 2m^{-i_\ell}< m^{-j}+\eps\,.
$$
As $\eps$ is arbitrary, $\pi_j(x)$ and $\pi_j(x')$ lie in $\Star(v,X_j)$ for some $j\in V_j$.

(2).  The proof is similar to (1). 
\end{proof}

\begin{corollary}
$\bar d_\infty$ is a distance function on $X_\infty$.
\end{corollary}
\begin{proof}
Suppose $x,x'\in X_\infty$,  $\bar d_\infty(x,x')=0$, and $i\in\Z$.   By Lemma \ref{lem-starscale},
for all  $j\geq i$ the set
$\{\pi_j^\infty(x),\pi_j^\infty(x')\}$ is contained in the star of some vertex $v\in V_j$.
Since $\pi_i^j:(X_j,d_j)\ra (X_i,d_i)$ is $1$-Lipschitz, it follows that $\{\pi_i^\infty(x),\pi_i^\infty(x')\}$
is contained in a set of $d_i$-diameter $\leq 2m^{-j}$.   Since $j$ is arbitrary, this means that
$\pi_i^\infty(x)=\pi_i^\infty(x')$.
\end{proof}

The following is a sharper statement:
\begin{lemma}
\label{lem-trimmedstarscale}
Suppose $x_1,x_2\in X_\infty$ are distinct points.    Let $j$ be the minimum of the indices $k\in \Z$ such that
$\{\pi_k(x_1),\pi_k(x_2)\}$ is not contained in the  trimmed star $\tstar(v,X_k)$ for any $v\in V_k$. 
Then 
\begin{equation}
\frac{(m-2)}{m}m^{-j}<\bar d_\infty(x_1,x_2)\leq 2m^{-(j-1)}\,.
\end{equation} 
\end{lemma}
\begin{proof}
The first inequality follows immediately from Lemma \ref{lem-starscale}.  By the choice of $j$, there
is a vertex $v\in V_{j-1}$ such that 
$$
\{\pi_{j-1}(x_1),\pi_{j-1}(x_2)\}\subset \tstar(v,X_{j-1})\subset\Star(v,X_{j-1})\,,
$$
 so
$
\bar d_\infty(x_1,x_2)\leq 2m^{-(j-1)}
$
   by Definition \ref{def-dbarinfty}.

\end{proof}

\subsection{A canonical map from the inverse limit to $\R$}
\label{subsec-canonical_map}
The next theorem contains  Theorem \ref{thm-easy_direction}.

\begin{theorem}
\label{thm-canonical_map_to_r}
Suppose $\{X_i\}$ is an admissible inverse system.
\begin{enumerate}
\setlength{\itemsep}{.5ex}
\item There is a compatible system
of direction preserving maps $\phi_i:X_i\ra \R$, such that for every $i$, the
restriction of $\phi_i$ to any edge $e\subset X_i$
 is a linear map onto a segment of length $m^{-i}$.  In particular, $\phi_i$ is
$1$-Lipschitz with respect to $d_i$.
 \item  
The system of maps $\{\phi_i:X_i\ra \R\}$ is unique up to  post-composition
with translation.
\item If $\phi:(X_\infty,\bar d_\infty)\ra \R$ is the map induced by $\{\phi_i\}$, then $\phi$ is $1$-Lipschitz,
and for every interval $I\subset\R$, the $\diam(I)$-components of 
$\phi^{-1}(I)$ have diameter at most $8m\cdot\diam(I)$.
\end{enumerate}
\end{theorem}
\begin{proof}

(1).  Let $X_{-\infty}$ denote the direct limit of the system $\{X_i\}$, i.e. 
$X_{-\infty}$ is the disjoint union $\sqcup_{i\in \Z}\;X_i$ modulo the 
equivalence relation that $X_i\ni x\sim x'\in X_j$  if and only 
if there is a $k\leq \min(i,j)$ such that $\pi_k^i(x)=\pi_k^j(x')$.
For every $i\in\Z$ there is a canonical projection map $\pi_{-\infty}^i:X_i\ra X_{-\infty}$.

If $k\in \Z$, then for all $i\leq k$ let $X_i^{(k-i)}$ denote
the $(k-i)$-fold iterated subdivision of $X_i$, as in Section \ref{subsec-subdivisions}.
Thus $\pi_i^j:X_j^{(k-j)}\ra X_i^{(k-i)}$ is simplicial for all $i\leq j\leq k$, and  restricts
to a direction-preserving isomorphism on each edge of $X_j^{(k-j)}$.  Therefore
the direct limit $X_{-\infty}$ inherits a directed graph structure, which we denote
$X_{-\infty}^{(k-\infty)}$, and for all $i\leq k$, the projection map
$\pi_{-\infty}^i:X_i^{(k-i)}\ra X_{-\infty}^{(k-\infty)}$ is simplicial, and 
a directed  isomorphism on each edge of $X_i^{(k-i)}$.    Condition
(3) of the definition of admissible systems implies that $X_{-\infty}^{(k-\infty)}$
is connected.

Note also that
for all $k \leq l$, the graph $X_{-\infty}^{(l-\infty)}$ is canonically
isomorphic to $(X_{-\infty}^{(k-\infty)})^{(l-k)}$.  In particular,
if  $v,v'$ are distinct vertices of $X_{-\infty}^{(k-\infty)}$,
then  their combinatorial
distance in $X_{-\infty}^{(l-\infty)}$  is at least $m^{l-k}$; morever  every 
vertex of $X_{-\infty}^{(l-\infty)}$
which is not a vertex of $X_{-\infty}^{(k-\infty)}$ must have valence $2$,
since it corresponds to an interior point of an edge of $X_{-\infty}^{(k-\infty)}$.   
It follows that $X_{-\infty}^{(k-\infty)}$
can contain at most one vertex $v$ which
has valence $\neq 2$.  Thus $X_{-\infty}^{(k-\infty)}$ is either
isomorphic to $\R$ with the standard subdivision, or to the union of
a single vertex $v$ with  a (possibly empty) collection of standard rays, each
of which 
is direction-preserving isomorphic to either $(-\infty,0]$ or $[0,\infty)$ with the
standard subdivision.    In either case, there is clearly a direction preserving
simplicial map  $X_{-\infty}^{(k-\infty)}\ra\R$ which is an isomorphism
on each edge of $X_{-\infty}^{(k-\infty)}$.  Precomposing this with the projection
maps $X_\infty\ra X_k\ra X_{-\infty}$ gives the desired maps $\phi_i$.  

(2).  Any such system $\{\phi_i:X_i\ra\R\}$ induces a map $\phi_{-\infty}:X_{-\infty}\ra\R$,
which for all $k\in \Z$ restricts to a direction preserving isomorphism on every edge
of $X_{-\infty}^{(k-\infty)}$.  From the description of $X_{-\infty}^{(k-\infty)}$, the map
$\phi_{-\infty}$ is unique up to post-composition with a translation.

(3).  If $x,x'\in X_\infty$ and $\{\pi_i(x),\pi_i(x')\}\subset \Star(v,X_i)$ for some $i\in \Z$,
$v\in V_i$, then by (1) 
$\{\phi(x),\phi(x')\}$ is contained in the union of two intervals
of length $m^{-i}$ in $\R$, and therefore
$d(\phi(x),\phi(x'))\leq 2m^{-i}$.  By the definition of $\bar d_\infty$, this 
implies that $d(\phi(x),\phi(x'))\leq \bar d_\infty(x,x')$ for all $x,x'\in X_\infty$, i.e.
$\phi$ is $1$-Lipschitz.

From the construction of the map $X_{-\infty}\ra \R$, there exists
a sequence $\{Y_i\}_{i\in\Z}$ of subdivisions of $\R$, such that $Y_{i+1}=Y_i^{(1)}$,
and $\phi_i:X_i\ra \R\simeq Y_i$ is simplicial and restricts to an isomorphism on 
every edge  of $X_i$.  

Now suppose $I\subset \R$ is an interval, and choose $i\in \Z$ such that
$\diam(I)\in[\frac{m^{-(i+1)}}{4},\frac{m^{-i}}{4})$.  
Then there is a vertex $v\in Y_i$ such that 
$I\subset \Star(v,Y_i)$ and $\dist(I,\R\setminus \Star(v,Y_i))>\diam(I)$.
Pick $x,x'\in \phi^{-1}(I)$ which lie in the same $\diam(I)$-component
of $\phi^{-1}(I)\subset X_\infty$, so there is a $\diam(I)$-path 
$x=x_0,\ldots,x_k=x'$ in $X_\infty$.   For each $\ell\in \{1,\ldots,k\}$ and every $\eps>0$,
there is a path $\ga_\ell$ in $X_i$ which joins $\pi_i(x_{\ell-1})$ to $\pi_i(x_\ell)$
such that $\length(\phi_i\circ\ga_\ell)\leq \diam(I)+\eps$.  When $\eps$ is sufficiently
small we get $\ga_\ell\subset \phi_i^{-1}(\Star^o(v,Y_i))$ 
because $\phi_i\circ \ga_\ell$ has  endpoints in $I$.
  Therefore $\pi_i(x),\pi_i(x')$
lie in the same path component of $\phi_i^{-1}(\Star^o(v,Y_i))$, which 
implies that they lie in $\Star(\hat v,X_i)$ for some vertex $v\in V_i$.
Hence $\bar d_\infty(x,x')\leq 2m^{-i}\leq 8m\cdot\diam(I)$.

\end{proof}

\bigskip\bigskip
\subsection{Directed paths, a partial ordering, and monotone paths}
\label{subsec-directedpaths}

Suppose $\{X_i\}$ is an admissible system, and $\{\phi_i:X_i\ra \R\}$ is a system of maps
as in Theorem \ref{thm-canonical_map_to_r}.   

\begin{definition}
A {\bf directed path} in $X_i$ is  a path
$\ga:I\ra X_j$ which is locally injective, and direction preserving (w.r.t. the usual direction on $I$).
A {\bf directed path} in $X_\infty$ is  a path $\ga:I\ra X_\infty$ such that $\pi_i\circ\ga:I\ra X_i$
is directed for all $i\in \Z$.
\end{definition}

If $\ga:I\ra X_j$ is a directed  path in $X_j$, then $\phi_j\circ\ga$ is a directed  path in $\R$, 
and hence it is embedded,  and has the same length as $\ga$.    
Therefore  $X_\infty$ and the $X_j$'s do not contain directed loops.
Furthermore, it follows that $\pi_i^j\circ\ga$ is a $d_i$-geodesic   in $X_i$ for all $i\leq j$.

\begin{definition}[Partial order]
We define a binary relation  on $X_i$, for $i\in \Z\cup\{\infty\}$
by declaring that $x\preceq y$ if there is a (possibly trivial) directed
path from  $x$ to $y$.   This defines a partial
order on $X_i$ since $X_i$ contains no directed loops.    As usual, $x\prec y$ means that $x\preceq y$
and $x\neq y$.  
\end{definition}

Since the projections $\pi_i^j:X_j\ra X_i$ are direction-preserving, they are order preserving for all
$i\leq j$, as is the projection map $\pi_i^\infty:X_\infty\ra X_i$.

\begin{lemma}
Suppose $\ga:I\ra X_\infty$ is a continuous map.
The following are equivalent:
\begin{enumerate}
\setlength{\itemsep}{.5ex}
\item $\ga$ is a directed geodesic, i.e. $\length(\ga)=d(\ga(0),\ga(1))$.
\item $\ga$ is a directed path.
\item $\pi_i\circ \ga$ is a directed path for all $i$.
\item $\phi\circ\ga:[0,1]\ra \R$ is a directed path.
\end{enumerate}
\end{lemma}
\begin{proof}
(1)$\implies$(2)$\implies$(3)$\implies$(4) is clear.

(4)$\implies$(3) follows from the fact that $\phi_i:X_i\ra\R$ restricts to a direction preserving isomorphism
on every edge of $X_i$.

(3)$\implies$(1).  For all $i\in \Z$, let $\ga_i\subset X_i$ be the union of the edges whose
interiors intersect the image of $\pi_i\circ\ga$.  Then $\ga_i$ is a directed edge path in $X_i$, and hence
$\phi_i\circ\ga_i$ is a directed edge path in $\R$ with the same number of edges.  
Therefore $\ga_i$ has
$<2+m^i(d(\phi(\ga(0)),\phi(\ga(1))))$ edges.   Since the vertices of $\ga_i$ belong to the image
of $\pi_i:X_\infty\ra X_i$, by the definition of $\bar d_\infty$, we have 
$\bar d_\infty(\ga(0),\ga(1))\leq 3m^{-i}+d(\phi(\ga(0)),\phi(\ga(1)))$.  Since $i$ is arbitrary we
get $\bar d_\infty(\ga(0),\ga(1))\leq d(\phi(\ga(0)),\phi(\ga(1)))$, and 
Theorem \ref{thm-canonical_map_to_r}(3) gives equality.  This holds for all subpaths of $\ga$
as well,  so $\ga$ is a geodesic.
\end{proof}

\begin{definition}
\label{def-monotone_geodesic}
A {\bf monotone geodesic segment in $X_i$} is the image of a directed isometric
embedding $\ga:[a,b]\ra X_i$; a {\bf monotone geodesic} in $X_i$ is the image of a 
directed isometric embedding $\R\ra (X_i,d_i)$.
A {\bf monotone geodesic
segment in $X_\infty$} is (the image of) 
a path $\ga:I\ra X_\infty$ satisfying any of the conditions of the lemma.
A {\bf monotone geodesic} is (the image of) a directed isometric embedding
$\R\ra X_\infty$, or equivalently, a geodesic $\ga\subset X_\infty$ which projects isometrically
under $\phi:X_\infty\ra\R$ onto $\R$.
\end{definition}

Monotone geodesics lead to monotone sets:
\begin{definition}
\label{def-monotone_set}
A subset $E\subset X_i$, $i\in \Z\cup\{\infty\}$, is {\bf monotone} if the characteristic
function $\chi_E$ restricts to a monotone function on any monotone geodesic
$\ga\subset X_i$ (i.e. $\ga\cap E$ and $\ga\setminus E$ are both connected subsets of $\ga$).
\end{definition}

\section{Inverse systems of graphs with path metrics}
\label{sec-inverse_systems_path_metrics}

For some admissible inverse systems, such as Examples \ref{ex-laaksodiamond}, \ref{ex-laaksoahlforsregularpi},
the path metrics $d_i$  induce a length structure on the inverse limit which is comparable
to $\bar d_\infty$.  We discuss
this special class here, comparing the length metric with the metric $\bar d_\infty$ defined
earlier.

In this section we assume that $\{X_i\}$ is an admissible inverse system satisfying two
additional conditions:
\begin{enumerate}[(a)]
\setlength{\itemsep}{.5ex}
\item $\pi_i:X_{i+1}\ra X_i'$ is an open map for all $i\in\Z$.
\item There is a $\th\in \N$ such that for every $i\in \Z$, $v\in V_i$, 
 $w,w'\in \pi_i^{-1}(v)$, the is  an edge path in $X_{i+1}$ 
 with at most $\th$ edges, which joins
$w$ and $w'$.
\end{enumerate}

Both conditions obviously hold in 
Examples \ref{ex-laaksodiamond}, \ref{ex-laaksoahlforsregularpi}.

\begin{lemma}[Path lifting]
\label{lem-pathlifting}

Suppose $0\leq i<j$,  $c:[a,b]\ra X_i$ is a path,  and $v\in (\pi_i^j)^{-1}(c(a))$.   Then 
there is a path  $\hat c:[a,b]\ra X_j$ such that
\begin{itemize}
\item $\hat c$ is a lift of  $c$:\quad $\pi_i^j\circ\hat c=c$.
\item 
$\hat c(a)=v$.
\end{itemize}
\end{lemma}
\begin{proof}
If $c$ is piecewise linear, and linear on each subinterval of the partition
$a=t_0<\ldots <t_k=b$,  then the existence of $\hat c$ follows by induction on $k$,
because of condition (a) above.  The general
case follows by approximation.
\end{proof}

\begin{lemma}

\mbox{}

\begin{enumerate}
\setlength{\itemsep}{.5ex}
\item $X_i$ is connected for all $i$.
\item $\pi_i^j:X_j\ra X_i$ and $\pi_i^\infty:X_\infty\ra X_i$ are surjective for all $i\leq j$.
\end{enumerate}

\end{lemma}
\begin{proof}
(1).  By Lemma \ref{lem-pathlifting} and condition (b), if $x_1,x_2\in X_i$ lie in the same
connected component of $X_i$, then any $x_1'\in \pi_i^{-1}(x_1)$, $x_2'\in \pi_i^{-1}(x_2)$
lie in the same connected component of $X_{i+1}$.  Iterating this, we get that
$(\pi_i^j)^{-1}(x_1)\cup(\pi_i^j)^{-1}(x_2)\subset X_j$ is contained in a single
component of $X_j$.   Now for every $j\in \Z$, $\hat x_1,\hat x_2\in X_j$, by 
Definition \ref{def-admissibleinversesystem} (3) there is an $i\leq j$ such that
$x_1=\pi_i(\hat x_1)$, $x_2=\pi_i(\hat x_2)$ lie in the same connected component
of $X_i$; therefore $\hat x_1,\hat x_2$ lie in the same component of $X_j$. 

(2).  $\pi_i:X_{i+1}\ra X_i$ is open by condition (a), $X_i$ is connected by (1), and 
$X_{i+1}$ is nonempty by Definition \ref{def-admissibleinversesystem} (1).  Therefore
$\pi_i:X_{i+1}\ra X_i$ is surjective.  It follows that $\pi_i^\infty:X_\infty\ra X_i$
is surjective as well.
\end{proof}

\bigskip

Note that $\pi_i:(X_{i+1},d_{i+1})\ra (X_i,d_i)$ is a $1$-Lipschitz map by Definition 
\ref{def-admissibleinversesystem}(2).

\begin{lemma}
\label{lem-fiberdiameter}

\mbox{}
\begin{enumerate}\setlength{\itemsep}{.5ex}
\item
For all $i\in\Z$, and every $x_1,x_2\in X_i$, $x_1',x_2'\in X_{i+1}$ with $\pi_i(x_k')=x_k$,
we  have
$$
d_{i+1}(x_1',x_2')\leq d_i(x_1,x_2)+2m^{-i}+\th\cdot m^{-(i+1)}\,.
$$
\item
If $i< j$, then for every $x_1,x_2\in X_i$, $x_1',x_2'\in X_j$ with $\pi_i(x_k')=x_k$,
we have 
\begin{equation}
\label{eqn-separationbound}
d_j(x_1',x_2')\leq d_i(x_1,x_2)+(2m+\th)\cdot\left(\frac{m^{-i}-m^{-j}}{m-1}  \right)\,.
\end{equation}
\end{enumerate}
\end{lemma}
\begin{proof}
(1).  Let $\ga:I\ra X_i$
be a path of length at most $d_i(x_1,x_2)+m^{-i}$ which joins $x_1$ to $x_2$,  and then 
continues to some vertex $v_2\in V_i$.  By Lemma \ref{lem-pathlifting} there is a lift 
$\ga':I\ra X_i$ starting at $x_1'$, and clearly $\length(\ga')=\length(\ga)$.
Since $x_2'$ has distance $<m^{-i}$ from $\pi_i^{-1}(v_2)$, (1) follows
from condition (b).

(2).  This follows by iterating (1).
\end{proof}

\bigskip\bigskip

As a consequence of Lemma \ref{lem-fiberdiameter}, the sequence of (pseudo)distance functions
$\{d_i\circ \pi_i^\infty:X_\infty\times X_\infty\ra [0,\infty)\}_{i\geq 0}$ converges 
geometrically to a distance function $d_\infty$
on $X_\infty$.    Since $\pi_i^\infty$ is surjective for all $i\geq 0$,
the  lemma also implies that $\{\pi_i^\infty:(X_\infty,d_\infty)\ra (X_i,d_i)\}_{i\geq 0}$
is a sequence of Gromov-Hausdorff approximations, so $(X_i,d_i)$ converges to
$(X_\infty,d_\infty)$ in the Gromov-Hausdorff topology.

\begin{lemma}

\mbox{}

\begin{enumerate}\setlength{\itemsep}{.5ex}
\item $\bar d_\infty\leq d_\infty$, with equality on montone geodesic segments.
\item $d_\infty\leq \frac{2m+\th}{2(m-1)}\cdot\bar d_\infty$.
\end{enumerate}
\end{lemma}
\begin{proof}
(1).  Suppose $x,x'\in X_\infty$ and for some $i\in \Z$ 
$\ga:I\ra X_i$ is a geodesic from $\pi_i(x)$ to $\pi_i(x')$.
Then the image of $\ga$ is contained in a chain of at most 
$2+\frac{d_\infty(x,x')}{2m^{-i}}$ stars in $X_i$.  Since
$\pi_i^\infty$ is surjective, Definition \ref{def-dbarinfty} implies 
$$
\bar d_\infty(x,x')\leq d_\infty(x,x')+2m^{-i}\,.
$$
Thus $\bar d_\infty\leq d_\infty$.

If $x,x'\in X_\infty$ are joined by a directed path $\ga:I\ra X_\infty$, then
\begin{align*}
d_\infty(x,x')\leq& \length(\ga)=\length(\pi_i\circ\ga)\\
=&d(\phi(\ga(0)),\phi(\ga(1)))
\leq\bar d_\infty(x,x')\,,
\end{align*}
where the last equality follows from Theorem \ref{thm-canonical_map_to_r}(3).

(2).    If $x,x'\in X_\infty$ and 
$\{\pi_i(x),\pi_i(x')\}\subset\Star(v,X_i)$ for some $i\in \Z$, $v\in V_i$, 
then $d_i(\pi_i(x),\pi_i(x'))\leq 2m^{-i}$ and so
$$
d_\infty(x,x')\leq \frac{(2m+\th)m^{-i}}{m-1}  
$$
by Lemma \ref{lem-fiberdiameter}.   It follows from Lemma \ref{lem-altdefdbarinfty}
that 
$$
d_\infty\leq \frac{(2m+\th)}{2(m-1)}\cdot \bar d_\infty\,.
$$

\end{proof}

\begin{corollary}
If
$\phi:X_\infty\ra\R$ is the map from 
Theorem \ref{thm-canonical_map_to_r}, then $\phi:(X_\infty,d_\infty)\ra\R$
satisfies the hypotheses of Theorem \ref{thm-length_space_embedding}.
\end{corollary}
\begin{proof}
This follows from the previous Lemma and Theorem \ref{thm-canonical_map_to_r}(3).
\end{proof}

\bigskip\bigskip

\section{Realizing metric spaces as  limits of admissible inverse systems}
\label{sec-realizing_metric_spaces}

 In this section, we 
characterize metric spaces which are bilipschitz homeomorphic to inverse limits
of admissible inverse systems,  proving Theorem \ref{thm-realization}. 

Suppose  a metric  space $Z$ is bilipschitz equivalent to 
the inverse limit of an admissible inverse system.   Evidently, if 
$X_\infty$ is such an inverse limit, $\phi:X_\infty\ra\R$
is as in Theorem \ref{thm-canonical_map_to_r}, and  $F:Z\ra X_\infty$ is
a bilipschitz homeomorphism, then the composition
$u=\phi\circ F:Z\ra \R$ has the property that for every interval $I\subset \R$, the
$\diam(I)$-components of $u^{-1}(I)$ have diameter at most
comparable to $\diam(I)$, (see Theorem \ref{thm-canonical_map_to_r}).  In other words
a necessary condition for a space to be bilipschitz homeomorphic to 
an inverse limit is the existence of a map satisfying the hypotheses of 
Theorem \ref{thm-realization}.     Theorem \ref{thm-realization} says that
the existence of such a map is sufficient.

We now prove Theorem \ref{thm-realization}.

Fix $m\in \N$, $m\geq 2$, and let $u:X\ra\R$ be as in the statement of the theorem.

Let $\{Y_i\}_{i\in\Z}$ be a  sequence of subdivisions of $\R$, where:
\begin{itemize}
\item 
$Y_i$ is a subdivision of $\R$ into intervals of length $m^{-i}$ for all $i\in \Z$.
\item $Y_{i+1}$ is a subdivision of $Y_i$ for all $i\in \Z$.
\end{itemize}

We define a simplicial graph $X_i$ as follows.  The vertex set of $X_i$ 
is the collection of pairs $(v,U)$ where $v$ is a vertex of $Y_i$ and $U$ is a $m^{-i}$-component of 
$u^{-1}(\Star(v,Y_i))$.    Two distinct vertices $(v_1,U_1),(v_2,U_2)\in X_i$ span
an edge iff  $U_1\cap U_2\neq \emptyset$; note that
this can only happen if $v_1,v_2$  are distinct adjacent vertices of $Y_i$. 

We have an projection map $\phi_i:X_i\ra Y_i$ which sends each vertex
$(v,U)$ of $X_i$ to $v\in Y_i$, and is a linear isomorphism
on each edge of $X_i$.   If $(\hat v,\hat U)$ is a vertex of $X_{i+1}$, 
there there will be a vertex $(v,U)$ of $X_i$ such that
$\hat U\subset U$ and $\Star(\hat v,Y_{i+1})\subset \Star(v,Y_i)$; 
there are at most two such vertices, and they will span an edge in $X_i$. 
Therefore we obtain a well-defined
projection map $\pi_i:X_{i+1}\ra X_i$ such that 
$\phi_i\circ \pi_i=\phi_{i+1}$, and which induces a simplicial map $\pi_i:X_{i+1}\ra X_i'$.

We define $f_i:X\ra X_i$ as follows.  Suppose $z\in X$ and $u(z)\in \R\simeq Y_i$ belongs to the
edge $e=\ol{v_1v_2}
=\Star(v_1,Y_i)\cap \Star(v_2,Y_i)$.   Then $z$ belongs to an $m^{-i}$-component
of $\Star(v_j,Y_i)$ for $j\in\{1,2\}$, and therefore these two components span an edge $\hat e$
of $X_i$ which is mapped isomorphically to $e$ by $\phi_i$.  We define $f_i(z)$
to be $\phi_i^{-1}(u(z))\cap\hat e$.
The sequence $\{f_i\}_{i\in \Z}$ is clearly compatible, so we have a well-defined
map $f_\infty:\ra X_\infty$.

Now suppose $z_1,z_2\in X$ and for some $i\in \Z$ we have $d(z_1,z_2)\leq m^{-i}$.
Then $\{u(z_1),u(z_2)\}\subset \Star(v,Y_i)$ for some vertex $v\in Y_i$, and $z_1,z_2$
lie in the same $m^{-i}$ component of $u^{-1}(\Star(v,Y_i))$.    Therefore
$\{f_i(z_1),f_i(z_2)\}$ is contained in $\Star(\hat v,X_i)$ for some vertex 
$\hat v$ of $X_i$.  It follows that $\bar d_\infty(f_\infty(z_1),f_\infty(z_2))\leq 2m^{-i}$,
from the definition of $\bar d_\infty$.    Thus $f_\infty$ is Lipschitz.

On the other hand, if $z_1,z_2\in X$ and $\bar d_\infty(f_\infty(z_1),f_\infty(z_2))\leq m^{-i}$,
then by Lemma \ref{lem-starscale}, we have $\{f_j(z_1),f_j(z_2))\}\subset \Star(\hat v,X_j)$ for some 
vertex $\hat v\in X_j$, where $|i-j|$ is controlled.  By the construction of $f_j$, this
means that $\{z_1,z_2\}$ lie in an $m^{-j}$-component of $u^{-1}(\Star(v,Y_j))$ for
$v=\phi_j(\hat v)\in Y_j$.  By our assumption on $u$, this gives $d(z_1,z_2)\lesssim m^{-i}$.
Thus $f_\infty$ is $L'$-bilipschitz, where $L'$ depends only on $C$ and $m$.

\section{A special case of Theorem \ref{thm-bilipschitzembedding} }
\label{sec-special_case}

Let $\{X_j\}$ be an admissible inverse system as in Definition \ref{def-admissibleinversesystem}.

\begin{assumption} 
\label{ass-tempfinite}
We will temporarily assume that:
\begin{enumerate}
\setlength{\itemsep}{.5ex}
\item $\pi_i$ is finite-to-one for all $i\in\Z$.
\item $X_i\simeq\R$ and $\pi_i:X_{i+1}\ra X_i'$ is an isomorphism for all $i\leq 0$.
\item For every $i\in\Z$,  every vertex $v\in V_i$ has neighbors $v_\pm\in V_i$
with $v_-\prec v\prec v_+$.   Equivalently, $X_i$ is a union of  (complete) monotone
geodesics (see Definition \ref{def-monotone_geodesic}). 
\end{enumerate}
In particular, (1) and (2) imply that $X_i$ has finite valence for all $i$.
\end{assumption}

\noindent
This extra assumption will be removed in
Section \ref{sec-generalcase}, in order to complete the proof in the general case.  We remark that it is possible
to adapt all the material  to the general setting, but this would impose a technical burden
that is largely avoidable.  Furthermore, Assumption \ref{ass-tempfinite} effectively
covers many cases of 
interest, such as Examples \ref{ex-laaksodiamond} and \ref{ex-laaksoahlforsregularpi}.

\section{Slices and the associated measures}
\label{sec-slices_measures}
Rather than working directly with monotone subsets as described in the introduction, we
instead work with subsets which we call slices, which are sets
of vertices which arise naturally as the boundaries of monotone subsets.     A slice
in $S\subset X_i$ gives rise to a family of slices in $X_{i+1}$ -- its children --
by performing  local modifications to
the inverse image $\pi_i^{-1}(S)\subset X_{i+1}$.   The children of $S$  carry a natural probability 
measure which treats disjoint  local modifications as independent.  This section develops the 
properties of slices and their children, and then introduces a family of measures $\{\Si_i'\}_{i\in\Z}$ on slices.

Let $\{X_i\}$ be an admissible inverse system satisfying Assumption \ref{ass-tempfinite}.

\subsection{Slices and their descendents}
We recall from Section \ref{subsec-directedpaths}  that $X_i$ carries a partial order $\preceq$.

\begin{definition}
\label{def-cutboundary}
A {\bf partial slice in $X_i'$} is a finite subset $S\subset V_i'$ which intersects
each monotone geodesic $\ga\subset X_i$ at most once; this is equivalent to saying
that no two elements of $S$ are comparable:  if $v,v'\in S$ and $v\preceq v'$, then
$v=v'$.
A {\bf slice in $X_i'$} is a partial slice which
intersects each monotone geodesic precisely once.
We denote the set of slices in $X_i'$ and partial slices in 
$X_i'$ by $\slice_i'$ and $\pslice_i'$ respectively.
\end{definition}

The vertex set  $V_i'$ is countable, in view of Assumption \ref{ass-tempfinite}.
Every partial slice is finite, so this implies that the collection of partial slices is countable.
When  $i\leq 0$,  $X_i'$ is a copy of $\R$ with a standard subdivision,  so the slices
in $X_i'$ are just singletons $\{v\}$, where $v\in V_i'$.

Note that we cannot have $w\prec x \prec w'$ for $x\in X_i$, $S\in \slice_i'$,
and $w,w'\in S$, because we could concatenate a monotone geodesic segment
joining $w$ to $w'$ with monotone rays, obtaining a monotone geodesic
which intersects $S$ twice.  Therefore we use the notation $x\prec S$ if there
is a $w\in S$ such that $x\prec w$.  The relations $x\succ S$, $x\preceq S$,
and $x \succeq S$ are defined similarly.  

A slice $S\in \slice_i'$ {\bf separates} (respectively {\bf weakly separates})
$x_1,x_2\in X_i$ if $x_1\prec S\prec x_2$
or $x_2\prec S\prec x_1$ (respectively $x_1\preceq S\preceq x_2$
or $x_2\preceq S\preceq x_1$).

If $S\in \slice_i'$, $v\in X_i\setminus S$, then we define
$$
\side(v,S)=
\begin{cases}
\prec &\quad \mbox{if} \quad v \prec S\\
\succ & \quad \mbox{if} \quad v \succ S
\end{cases}
$$

Slices give rise to monotone sets:
\begin{lemma}
\label{lem-cutboundariesgivemonotonecuts}
If $S\in \slice_i'$, define 
$
S_{\preceq}=\{x\in X_i\mid x\preceq S\}$ and
$S_{\succeq}=\{x\in X_i\mid x\succeq S\}
$.  Then $S_{\preceq}$ and $S_{\succeq}$ are
both monotone sets with boundary $S$.
\end{lemma}
\begin{proof}
Since $S_\preceq^c=S_\succ=\{x\in X_i\mid x\succ S\}$, the monotonicity of $S_\preceq$ 
follows immediately from the definition of slices.  Similarly for $S_\succeq$.
\end{proof}

Given a vertex $v\in V_i$, we can associate a collection of partial slices
 in $X_i'$:
\begin{definition}
If $v\in V_i$, a {\bf child of $v$} is a maximal partial slice  $S'\in \pslice_i'$
which is contained in the trimmed star $\tstar(v,X_i)$, see Figure
\ref{fig-children}.  In other words, 
$S'\subset \tstar(v,X_i)$ and precisely one of the
following holds:
\begin{enumerate}\setlength{\itemsep}{.5ex}
\item $S'=\{v\}$.
\item For every vertex $w\in\Star(v,X_i)$ with $v\prec w$, $S'$ intersects
the edge $\ol{vw}$ in precisely one point, which is an interior point.
\item For every vertex $w\in\Star(v,X_i)$ with $w\prec v$, $S'$ intersects
the edge $\ol{vw}$ in precisely one point, which is an interior point.
\end{enumerate}
We denote the collection of children of $v$ by $\child(v)$, and refer to the
children  in the above cases as {\bf children of type (1), (2), or (3)} respectively.
\end{definition}

\begin{figure}[h] 

\begin{center}  
\includegraphics[scale=.8]{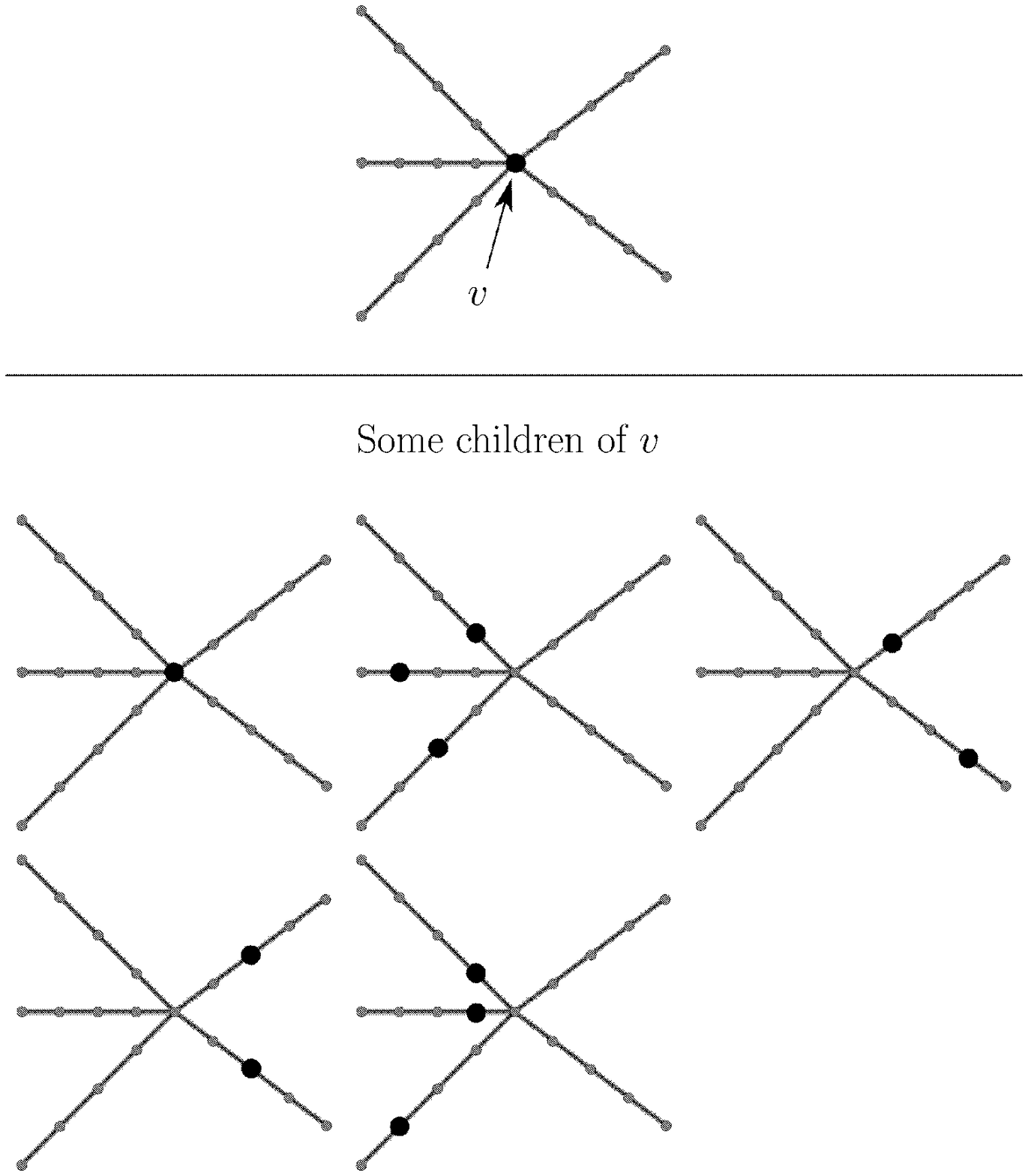} 
\caption{\label{fig-children}}
\end{center}

\end{figure}

Note that if $S\in \pslice_i'$ and $v_1,v_2\in V_{i+1}$ are distinct vertices
lying in $\pi_i^{-1}(S)$, then their trimmed stars are disjoint.

\begin{definition}
If $S\in \pslice_i'$ is a partial slice, a {\bf child of $S$}
is a subset  $S'\subset V_{i+1}'$  
obtained by replacing each vertex $v\in \pi_i^{-1}(S)\subset V_{i+1}$ with one
of its children, so that $S'$ is a subset of $V_{i+1}'$.  More formally,
$S'$ belongs to the image of the ``union map'' 
$$
\prod_{v\in \pi_i^{-1}(S)}\child(v)\lra V_{i+1}'
$$
which sends $\prod_v\,(S_v)$ to $\cup_v\,S_v$.  We use $\child(S)\subset \pslice_{i+1}'$ to denote the
children of $S\in \pslice_i'$.
\end{definition}

\begin{lemma}
If $S$ is a partial slice, so is
each of its children.    Moreover, if $S\in \slice_i'$ is a slice, so is $S'$.
\end{lemma}
\begin{proof}
Suppose $S'$ is  a child of the partial slice $S\in \pslice_i'$, and $\ga'\subset X_{i+1}$
is a monotone geodesic.    Then $\ga'$ projects isomorphically to a monotone geodesic
$\ga\subset X_i$, so $\pi_i^{-1}(S)\cap \ga'$ contains at most one vertex $v'\in V_{i+1}$.
From the definition of children, it follows that $\ga'\cap S'$ contains at most one point.
If $S$ is a slice, then $S'$ contains precisely one vertex $v\in V_{i+1}'$, and therefore
$S'$ contains a child of $\{v\}$, which will intersect $\ga'$ in precisely one point.

\end{proof}

\begin{definition}
If $S\in\pslice_i'$ and $j>i$, then a partial slice
$S'\in \pslice_j'$ is a {\bf descendent of $S$ in $X_j'$} if
there exist $S=S_i, S_{i+1},\ldots,S_j=S'$
such that  for all $k\in\{i+1,\ldots,j\}$, $S_k\in \pslice_k'$ and $S_k$ is a child of 
$S_{k-1}$; in other words, $S'$ is an iterated child of $S$. 
We denote the collection of such 
descendents by $\desc(S,X_j')$.
\end{definition}

\begin{lemma}
\label{lem-projectionhausdorffdistance}
 For all $i<j$, if $S'\in \pslice_j'$ is a descendent of $S\in \pslice_i'$, then 
 $\pi_i(S')\subset \cup_{w\in S}\Star^o(w,X_i')$.
\end{lemma}
\begin{proof}
  Suppose $S=S_i,\ldots,S_j=S'$, where $S_k\in \pslice_k'$ and $S_{k+1}$
is a child of $S_k$ for all $i\leq k<j$.  Then 
$$
\pi_k(S_{k+1})\subset \pi_k\left(\cup_{w\in \pi_k^{-1}(S_k)}\tStar(w,X_{k+1}')\right)\,.
$$
Iterating this yields the lemma.
\end{proof}

\subsection{A measure on slices}

We now define a measure on $\slice_i'$ for each $i$, by an iterated diffusion construction.
To do this, we first associate with each vertex $v\in V_i$ a probability measure on 
its children $\child(v)\subset \pslice_i'$.

\begin{definition}
\label{def-vertexchildmeasure}
If $v\in V_i$, let $w_{\child(v)}$ be the probability measure on $\child(v)$
which:
\begin{itemize}
\setlength{\itemsep}{.5ex}
\item Assigns measure $\frac1m$ to the child $\{v\}\in \child(v)$ of type (1).
\item 
Uniformly distributes measure $\frac12\cdot \frac{(m-1)}{m}$ among the children of type (2).
Equivalently, for each vertex $\hat v\in V_i$  adjacent to $v$ with $v\prec \hat v$, we take the uniform
measure on the $(m-1)$ vertices in $V_i'$ which are interior points of the edge $\ol{v\hat v}$, take the 
product of these measures as $\hat v$ ranges over 
$$
\{\bar v\in V_i\mid \bar v\in \Star(v,X_i),\;v\prec\bar v\}\,
$$
and then multiply the result by $\frac12\cdot \frac{(m-1)}{m}$.
\item Uniformly distributes measure $\frac12\cdot \frac{(m-1)}{m}$ among the children of type (3).
\end{itemize}
\end{definition}
Note that if $v'\in V_i'$ belongs to the trimmed star of $v$, then the $w_{\child(v)}$ measure
of the children of $v$ which contain $v'$ is 
\begin{equation}
\label{eqn-childmeasurebalanced}
w_{\child(v)}(\{S\in \child(v)\mid v'\in S\})=
\begin{cases}
\frac{1}{m} &\quad \mbox{if} \quad v'=v\\
\frac{1}{2m} & \quad \mbox{if} \quad v'\neq v
\end{cases}
\end{equation}

Using the measures $w_{\child(v)}$ we define a measure on the children of a
 slice:
\begin{definition}
If $S\in\slice_i'$, we define a probability measure $K_S$ on $\child(S)$ as follows.
We take the product measure $\prod_{v\in \pi_i^{-1}(S)}\,w_{\child(v)}$ on 
$\prod_{v\in \pi_i^{-1}(S)}\,\child(v)$, and push it forward under the
union map 
$$
\prod_{v\in \pi_i^{-1}(S)}\,\child(v)\ra \slice_{i+1}'\,.
$$
 In probabilistic language,
for each $v\in \pi_i^{-1}(S)$, we independently choose a child of $v$ according to the distribution
$w_{\child(v)}$, 
 and then take the union of the resulting children.  Note that this is well-defined because
 the inverse image of any slice is nonempty.
\end{definition}

Now given a measure $\nu$ on $\slice_i'$, we diffuse it to a measure
$\nu'$ on $\slice_{i+1}'$:
\begin{equation}
\label{eqn-diffusion}
\nu'=\sum_{S\in\slice_i'}\,K_{S}\,\nu(S)\,.
\end{equation}
If we view the collection $\{K_S\}_{S\in \slice_i'}$ as defining a kernel 
$$
K_i:\slice_i'\times 
\slice_{i+1}'\ra [0,1]
$$
by the formula $K_i(S,S')=K_S(S')$, then the associated diffusion operator $K_i$ is given by
\begin{equation}
\label{eqn-kernel}
K_i(\nu)(S')=\sum_{S\in \slice_i'}K(S,S')\nu(S)\,.
\end{equation}
When  $i<0$, then this sum will be finite for any measure $\nu$ since $K(S,S')\neq 0$ for only
finitely many $S$. 
\begin{lemma}
\label{lem-finitesum}
 When $i\geq 0$, the sum will be finite provided $\nu$ is supported on the
descendents of slices in $X_0'$.
\end{lemma}
\begin{proof}
For a given $S'\in \slice_{i+1}'$, the summand $K(S,S')\nu(S)$ is nonzero 
only if $S$ is a descendent of a slice $\{v\}\in V_0'$ and $S'$ is a child of $S$.  
By Lemma \ref{lem-projectionhausdorffdistance} this means that 
$\pi_0(S')\subset \Star^o(v,X_0')$, 
so there are only finitely many possibilities for such $S$.
\end{proof}

\begin{definition}
For $i\leq 0$, let  $\Si_i'$ the measure on $\slice_i'$ which assigns measure $m^{-(i+1)}$ to 
each slice in $\slice_i'=V_i'$.   For $i>0$ we define a measure $\Si_i'$ 
on $\slice_i'$ inductively by 
$\Si_i'=K_{i-1}(\Si_{i-1}')$.  This is well-defined by Lemma \ref{lem-finitesum}.
\end{definition}

For every $S\in \slice_i'$ and every $j>i$, 
we may also obtain a well-defined probability measure
on $\slice_j'$ which is supported on the descendents of $S$, by the formula
\begin{equation}
\label{eqn-descendentmeasure}
K_{j-1}\circ\ldots\circ K_i(\delta_S)\,,
\end{equation}
where $\delta_S$ is a Dirac mass on $S$.  Using this probability measure, we may speak of 
the measure of descendents of $S\in\slice_j'$.

\section{Estimates on the family of measures $\{\Si_i'\}_{i\in\Z}$}
\label{sec-estimates}

In this section we will prove (mostly by induction arguments) several estimates on the slice/cut
measures $\{\Si_i'\}$ and cut metrics $\{d_{\Si_i'}\}$ that will be needed in Section 
\ref{sec-completion_assumption}.

We first observe that the slices passing through a vertex $v\in V_i'$ have measure
$m^{-(i+1)}$:

\begin{lemma}
\label{lem-normalized}
For all $i\in\Z$, and every $v\in V_i'$, the $\Si_i'$-measure of the collection of slices containing
$v$ is precisely $m^{-(i+1)}$:
$$
\Si_i'(\{S\in \slice_i'\mid v\in S\})=m^{-(i+1)}\,.
$$ 
\end{lemma}

\begin{proof}
When $i\leq 0$ this reduces to the definition of $\Si_i'$.   So pick $i>0$,  $v\in V_i'$, and assume 
inductively that the lemma is true for $i-1$.  

{\em Case 1. $v\in V_i$.}  In this case, if a slice $S\in \slice_{i-1}'$ has a child
$S'\in \slice_i'$ containing $v$, then $\pi_{i-1}(v)\in S$. 
By Definition 
\ref{def-vertexchildmeasure},  for such an $S$, the 
fraction of its children containing $v$ is precisely $\frac1m$.    Therefore by the induction hypothesis we have 
$$
\Si_i'(\{S'\mid v\in S'\})
=\frac1m\,\Si_{i-1}'(\{S\in \slice_{i-1}'\mid \pi_{i-1}(v)\in S\})=m^{-(i+1)}\,.
$$

{\em Case 2. $v\not\in V_i$.}  Then $v$ belongs to a unique edge $\ol{w_1w_2}\subset X_i$,
where $w_1,w_2\in V_i$.   In this case, a slice $S\in \slice_{i-1}'$ has a child
$S'\in \slice_i'$ containing $v$ if and only if $S$ contains $\pi_{i-1}(w_1)$
or $\pi_{i-1}(w_2)$.  Since these possibilities are mutally exclusive (from the definition 
of slice), and each contributes a measure $\frac12 m^{-(i+1)}$ by the induction hypothesis
and Definition \ref{def-vertexchildmeasure}, the lemma follows.
\end{proof}

Recall that by Lemma \ref{lem-cutboundariesgivemonotonecuts}, for every $S\in \slice_i'$ the subset
$S_\leq=\{x\in X_i\mid x\leq S\}$
is a monotone subset of $X_i$.

\begin{definition}
\label{def-dsii'}
Viewing $\Si_i'$ as a cut measure  on $X_i$ via the identification $S\longleftrightarrow S_\leq$,  
we let $d_{\Si_i'}$ denote the corresponding cut metric
on $X_i$.  Equivalently, for $x_1,x_2\in X_i$, 
$$
d_{\Si_i'}(x_1,x_2)=\sum_{S\in \slice_i'}d_{S_\leq}(x_1,x_2)\,\Si_i'(S)\,
$$
where 

\begin{align*}
d_{S_\leq}(x_1,x_2)=&|\chi_{S_\leq}(x_1)-\chi_{S_\leq}(x_2)|\\
&{}\\
=&
\begin{cases}
1 &\quad \mbox{if} \quad x_1\preceq S\prec x_2\quad\mbox{or}\quad x_2\preceq S\prec x_1\\
0 & \quad \mbox{otherwise} 
\end{cases}
\end{align*}
\end{definition}

\begin{lemma}
\label{lem-distancechange}
If $x_1',x_2'\in X_{i+1}$, and $\pi_i(x_j')=x_j\in X_i$, then 
\begin{equation}
\label{eqn-distancechange}
|d_{\Si_{i+1}'}(x_1',x_2')-d_{\Si_i'}(x_1,x_2)|\leq 4\,m^{-(i+1)}\,.
\end{equation}
\end{lemma}
\begin{proof}
For $j\in \{1,2\}$ let $S_j$ be the collection of slices $S\in \slice_i'$
which have a child $S'\in \slice_{i+1}'$ such that $\side(x_j,S)\neq \side(x_j',S')$. 
From the definition of children, it follows that if $S\in S_j$, then $x_j$ lies in
$\tstar(v,X_i')$ for some $v\in S$.  Thus, if $\ol{w_1w_2}$ is an edge of $X_i'$ containing
$x_j$, then $S_j\subset\{S\in \slice_i'\mid S\cap \{w_1,w_2\}\neq\emptyset\}$.
By Lemma \ref{lem-normalized}, we have $\Si_i'(S_j)\leq 2m^{-(i+1)}$.  Now by the 
definition of the cut metrics, we get
$$
|d_{\Si_{i+1}'}(x_1',x_2')-d_{\Si_i'}(x_1,x_2)|\leq \Si_i'(S_1\cup S_2)\leq 4\,m^{-(i+1)}\,.
$$
\end{proof}

\begin{lemma}[Persistence of sides]
\label{lem-persistenceofsides}
\quad Suppose $x\in X_i$, $S\in \slice_i'$, and $x\not\in \cup_{v\in S}\;\Star^o(v,X_i')$.
Then for every   $j\geq i$,  every $x'\in X_j$  with $\pi_i(x')=x$, and 
for every descendent $S'\in \slice_j'$ of $S$,
we have 
$$
x' \not\in\cup_{v\in S'}\;\Star^o(v,X_j') \quad\mbox{and}\quad
\side(x',S')=\side(x,S)\,.
$$
\end{lemma}
\begin{proof}
First suppose $j=i+1$.   
Then $x\notin\cup_{v\in \pi_i^{-1}(S)}\Star^o(v,X_{i+1})$ because $\pi_i:X_{i+1}\ra X_i'$
is a simplicial mapping.  
Clearly this implies $\side(x',S')=\side(x',\pi_{i+1}^{-1}(S))=\side(x,S)$.

The $j>i+1$ case now follows by induction.

\end{proof}

\begin{lemma}[Persistence of separation]
\label{lem-persistenceofseparation}
\quad
There is a constant $A=A(m)\in (0,1)$ with the following property.
Suppose  $j>i$, $x_1,x_2\in X_i$,  $x_1',x_2'\in X_j$, and $\pi_i(x_1')=x_1$,
$\pi_i(x_2')=x_2$.   Suppose in addition that
\begin{itemize}
\item 
 $S\in \slice_i'$ is a slice which weakly separates
$x_1$ and $x_2$.
\item
$x_2\notin \cup_{v\in S}\Star^o(v,X_i')$.
\end{itemize}
Then the measure of the collection of descendents $S'\in
\desc(S,X_j')$  which separate $x_1'$ and $x_2'$
is at least $A$; here we refer to the probability measure on 
$\desc(S,X_j)$ that was defined in equation (\ref{eqn-descendentmeasure}).
\end{lemma}
\begin{proof}
Since $S$ weakly separates $x_1$ and $x_2$ but $x_2\notin S$, 
without loss of generality we may assume that $x_1\preceq S\prec x_2$, since the 
case $x_2\prec S\preceq x_1$ is similar. 

If 
$x_1\notin \cup_{v\in S}\Star^o(v,X_i')$, then by Lemma \ref{lem-persistenceofsides}, for all
$S'\in \desc(S,X_j')$ we have $x_1'\prec S'\prec x_2'$, so we are done in this case.

Therefore we assume that  there exist $v\in  S$ and $v'\in \pi_i^{-1}(v)$
such that $x_1\in \Star^o(v,X_i')$, $x_1'\in \Star(v',X_{i+1})$, and 
$x_1'\preceq v'$.  If $S'\in \slice_{i+1}'$ is a child of $S$ containing
a child of $v'$ of type (2), then 
$x_1'\notin \cup_{w\in S'}\Star^o(w,X_{i+1}')$; moreover
the collection of such slices $S'$ form a fraction at least $\frac12 \frac{m-1}{m}$ of the 
children of $S$.  Applying Lemma \ref{lem-persistenceofsides} to each such slice $S'$,
we conclude that for every $S''\in \desc(S',X_j')$, we have $x_1'\prec S''\prec x_2'$. 
This proves the lemma. 
\end{proof}

\begin{lemma}
\label{lem-notinsametrimmedstar}
Suppose $x_1,x_2\in X_i$, $j>i$, $x_1',x_2'\in X_j$, $\pi_i(x_1')=x_1$, $\pi_i(x_2')=x_2$,
and $\{x_1,x_2\}$ is not contained in the trimmed star of any vertex $v\in V_i$.
Then $d_{\Si_j'}(x_1',x_2')\geq Am^{-(i+2)}$, where $A$ is the constant from 
Lemma \ref{lem-persistenceofseparation}.
\end{lemma}
\begin{proof}
Choose  $v_1\in V_i$ such that $x_1\in \tstar(v_1,X_i)$.

Observe that of the 
children $W$ of $v_1$, a measure at least $\frac{1}{2m}$ lie weakly on each side of $x_1$
and satisfy
$x_2\notin\cup_{v\in W}\Star^o(v,X_i')$,  in view of our assumption
on $x_1$ and $x_2$, i.e.
$$
w_{\child(v_1)}\left( 
\{ W\in \child(v_1)\mid W\preceq x_1,\;
x_2\notin\cup_{v\in W}\Star^o(v,X_i')\}  \right)\geq \frac{1}{2m}\,,
$$
$$
w_{\child(v_1)}\left( 
\{ W\in \child(v_1)\mid W\succeq x_1,\;
x_2\notin\cup_{v\in W}\Star^o(v,X_i')\}  \right)
\geq \frac{1}{2m}\,.
$$  

Suppose $S\in \slice_{i-1}'$ and $v_1\in \pi_{i-1}^{-1}(S)$.
Then each child $S'\in \slice_i'$ of $S$ contains some child of $v_1$,
and $\side(x_2,S')$ is independent of this choice, because $x_2$ lies outside
$\tstar(v_1,X_i)$.   Furthermore, if $x_2\in \Star(v_2,X_i)$ for some $v_2\in (V_i\cap \pi_{i-1}^{-1}(S))
\setminus\{v_1\}$, then $S'$ contains a child of $v_2$, and a fraction at least $\frac1m$
of this set of children $W\in \child(v_2)$
satisfies $x_2\notin \cup_{v\in W}\Star^o(v,X_i')$.
Thus a fraction at least $\frac{1}{2m^2}$ of the children of $S$ satisfy the assumptions
of Lemma \ref{lem-persistenceofseparation}.

Since the set of $S\in \slice_{i-1}'$ with $v_1\in \pi_{i-1}^{-1}(S)$ has $\Si_{i-1}'$-measure
$m^{-i}$ by Lemma \ref{lem-normalized}, by the preceding reasoning, we conclude that 
$d_{\Si_j'}(x_1',x_2')\geq A\,m^{-(i+2)}$.
\end{proof}

\begin{lemma}
\label{lem-dsiedge}

Suppose $i,j\in \Z$, $i\leq j$, $x_1,x_2\in X_j$, $e$ is an edge of $X_i$, and $\pi_i(x_1),\pi_i(x_2)\in e$.
Then 
\begin{equation}
\label{eqn-dsij}
d_{\Si_j'}(x_1,x_2)\leq m^{-i}\,.
\end{equation}
\end{lemma}
\begin{proof}
Let $v_1,v_2$ be the endpoints of $e$, where $v_1\prec v_2$.   We may assume without loss of 
generality that $\pi_i(x_1)\preceq \pi_i(x_2)$.   

By Definition \ref{def-dsii'}, the distance $d_{\Si_j'}(x_1,x_2)$ is the 
$\Si_j'$-measure of the set 
$$
Y=\{S\in\slice_j'\mid \;x_1\preceq S\prec x_2\quad \mbox{or}\quad
x_2\preceq S \prec x_1\;\}.
$$

If  $S\in \slice_i'$ and $S\cap e=\emptyset$, then $S$ does not weakly separate $\pi_i(x_1),\pi_i(x_2)$, so by 
Lemma \ref{lem-persistenceofsides},  no descendent $S'\in \desc(S,X_j')$
can weakly separate $x_1$ and $x_2$, i.e. $Y\cap \desc(S,X_j')=\emptyset$. 
For $v\in V_i'$ let  
$$
\slice_i'(v)=\{S\in \slice_i'\mid v\in S\}\,
$$
 and let $\Si_i'\on\slice_i'(v)$
be the restriction of $\Si_i'$ to $\slice_i'(v)$.
Thus, using the diffusion operators $K_l$ from (\ref{eqn-diffusion}), the above observation implies that
\begin{align*}
\Si_j'(Y)=&K_{j-1}\circ\ldots\circ K_i(\Si_i')(Y)\\
=&\sum_{v\in e\cap V_i'}K_{j-1}\circ\ldots\circ K_i(\Si_i'\on\slice_i'(v))(Y)
\end{align*}
\begin{equation}
\label{eqn-breakdown}
\quad \quad\leq (m-1)m^{-(i+1)}+\sum_{v\in \{v_1,v_2\} }K_{j-1}\circ\ldots\circ K_i(\Si_i'\on\slice_i'(v))(Y)\,,
\end{equation}
because the mass of $\Si_i'\on \slice_i'(v)$ is $m^{-(i+1)}$ by 
Lemma \ref{lem-normalized}.    The remainder of the proof
is devoting to showing that the total contribution from the last two terms in 
(\ref{eqn-breakdown}) is at most $m^{-(i+1)}$.

Suppose $\pi_i(x_1)\notin\Star^o(v_1,X_i')$.
If  $S\in\slice_i'(v_1)$, then no descendent $S'\in\desc(S,X_j')$
can weakly separate $x_1,x_2$ by Lemma \ref{lem-persistenceofsides}.
Therefore
$$
K_{j-1}\circ\ldots\circ K_i(\Si_i'\on\slice_i'(v_1))(Y)=0\,,
$$
so we are done in this case.  
Hence we may
assume that $\pi_i(x_1)\in\Star^o(v_1,X_i')$, and by similar reasoning, that 
$\pi_i(x_2)\in\Star^o(v_2,X_i')$.   This implies by Lemma \ref{lem-persistenceofsides}
that if $v_1\in S\in \slice_i'$ and $S'\in \desc(S,X_j')$ then $S'\prec x_2$; similarly, 
if $v_2\in S\in \slice_i'$ and $S'\in \desc(S,X_j')$, then $x_1\prec S'$.

By (\ref{eqn-breakdown}), the lemma now reduces to   the following two claims:

{\em Claim 1.}  $K_{j-1}\circ\ldots\circ K_i(\Si_i'\on\slice_i'(v_1))(Y)\leq \frac12(m^{-(i+1)}+m^{-(j+1)})$.

{\em Claim 2.}  $K_{j-1}\circ\ldots\circ K_i(\Si_i'\on\slice_i'(v_2))(Y)\leq \frac12(m^{-(i+1)}-m^{-(j+1)})$.

\noindent
{\em Proof of Claim 1.}
For each $i\leq k\leq j$, let $w_k'$ be the unique vertex in $V_k'$ such that 
$\pi_k(x_1)\in \Star^o(w_k',X_k')$ and $w_k'\preceq \pi_k(x_1)$.  Likewise,
let $w_k$ be the unique vertex in $V_k$ such that 
$\pi_k(x_1)\in \Star^o(w_k,X_k)$ and $w_k\preceq \pi_k(x_1)$.   Thus $w_k'\in \Star^o(w_k,X_k)$.  Let 
$k_0$ be the maximum of the integers $k\in[i,j]$ such that $\pi_i(w_{\bar k}')=v_1$ for all
$i\leq \bar k\leq k$.  It follows that $w_k'=w_k$ for all $i\leq k\leq k_0$.

For $i\leq k\leq j$, let $A_k$ be the collection of slices $S\in \slice_k'$
which contain $w_k$, and let $B_k$ be the collection
of slices $S\in \slice_k'$ which contain a child of $w_k$ of  type (3).   We now define a sequence
of measures $\{\al_k\}_{i\leq k\leq k_0}$ inductively as follows.  Let $\al_i$ be the restriction of 
$\Si_i'$ to $\{S\in \slice_i'\mid v_1\in S\}$.   For $k<k_0$, we define $\al_{k+1}$ to the 
restriction of $K_k\al_k$ to $\slice_{k+1}'\setminus B_{k+1}$, where $K_k$ is the diffusion 
operator (\ref{eqn-diffusion}).    Since by Lemma \ref{lem-persistenceofsides} 
the set of descendents $\desc(B_k,X_j')$ is disjoint from $Y$,
it follows that 
\begin{equation}
\label{eqn-sameassij'}
\quad \quad K_{j-1}\circ\ldots\circ K_i(\Si_i'\on\slice_i'(v_1))\on Y
=(K_{j-1}\circ\ldots K_k\al_k)\on Y
\end{equation}
for all $k\leq k_0$.

Note that by the definition of the diffusion operator $K_{k-1}$ we have
\begin{align}
\label{eqn-recursionestimates1}
\al_k(A_k)\geq& \frac1m\cdot\al_{k-1}(A_{k-1})\\
\label{eqn-recursionestimates2}(K_{k-1}\al_{k-1})(B_k)\geq& \frac{(m-1)}{2m}\cdot\al_{k-1}(A_{k-1})
\end{align}
for all $i<k\leq k_0$.
This yields $\al_k(A_k)\geq m^{-(k+1)}$ for all $i\leq k\leq k_0$. 
Hence for all $i<k\leq k_0$ we get
\begin{align*}
\al_k(\slice_k')=&\, \al_{k-1}(\slice_{k-1}')-(K_{k-1}\al_{k-1})(B_k)\\
\leq& \,\al_{k-1}(\slice_{k-1}')-\frac{(m-1)}{2m}\cdot m^{-k}
\end{align*}
and so
\begin{align*}
\al_{k_0}(\slice_{k_0}')\leq& m^{-(i+1)}-\frac{(m-1)}{2m}\cdot(m^{-(i+1)}+\ldots +m^{-k_0})\\
=&\frac12 m^{-(i+1)}+\frac12 m^{-(k_0+1)}\,.
\end{align*}
This gives Claim 1 when $k_0=j$, by  (\ref{eqn-sameassij'}).

We now assume that
$k_0<j$.    Then $w_{k_0+1}\prec w_{k_0+1}'$.   By Lemma \ref{lem-persistenceofsides}, every 
descendent $S'\in\slice_j'$ of a slice $S\in A_{k_0+1}\cup B_{k_0+1}$ will satisfy $S'\prec x_1$, 
so $S'\notin Y$.  Therefore if we define $\al_{k_0+1}$ to be the restriction of 
$K_{k_0}\al_{k_0}$ to $\slice_{k_0+1}'\setminus (A_{k_0+1}\cup B_{k_0+1})$, then 
$$
K_{j-1}\circ\ldots\circ K_i(\Si_i'\on\slice_i'(v_1))(Y)
=K_{j-1}\circ\ldots\circ K_{k_0+1}(\al_{k_0+1})(Y)\,.
$$
Also,
as in (\ref{eqn-recursionestimates1})--(\ref{eqn-recursionestimates2}), we get 
\begin{align*}
K_{k_0}\al_{k_0}(A_{k_0+1}\sqcup B_{k_0+1})\geq& \left(\frac1m +\frac{(m-1)}{2m}\right)\al_{k_0}(A_{k_0})\\
\geq& \left(\frac1m +\frac{(m-1)}{2m}\right)m^{-(k_0+1)}\,.
\end{align*}
Therefore
\begin{align*}
K_{j-1}\circ\ldots\circ K_{k_0+1}(\al_{k_0+1})(Y)\leq&\;
\al_{k_0+1}(\slice_{k_0+1}')\\
= &\;\al_{k_0}(\slice_{k_0}')
-K_{k_0}\al_{k_0}(A_{k_0+1}\sqcup B_{k_0+1})\\
\leq &\;\frac12 m^{-(i+1)}+\frac12 m^{-(k_0+1)}
-\left(\frac1m +\frac{(m-1)}{2m}\right)m^{-(k_0+1)}\\
\leq &\;\frac12 m^{-(i+1)}-\frac12 m^{-(k_0+1)}
\end{align*}
so Claim 1 holds.

\bigskip
\noindent
{\em Proof of Claim 2.}  The proof is similar to that of Claim 1, except that one replaces $v_1$
with $v_2$, and reverses the orderings.  However, in the case when 
$k_0=j$, one simply notes that any slice $S'\in A_{k_0}=A_j$ satisfies
$x_1\preceq S'$, $x_2\preceq S'$, so $S'\notin Y$.   Therefore we may remove the measure
contributed by $A_{k_0}$ from our estimate, making it smaller by 
$m^{-(j+1)}$.

\end{proof}

\begin{corollary}
\label{cor-dsistar}
Suppose $i,j\in \Z$, $i\leq j$, $x_1,x_2\in X_j$, $v\in V_i$, and
$\{x_1,x_2\}\subset (\pi_i^j)^{-1}(\Star(v,X_i))$.
Then 
\begin{equation}
\label{eqn-dsij'}
d_{\Si_j'}(x_1,x_2)\leq 2m^{-i}\,.
\end{equation}
\end{corollary}
\begin{proof}
First suppose there is an $x\in X_j$ such that $\pi_i^j(x)=v$.   Then $\{\pi_i^j(x),\pi_i^j(x_1)\}$
lies in an edge of $X_i$, so by Lemma \ref{lem-dsiedge} we have
$d_{\Si_j'}(x,x_1)\leq m^{-i}$, and similarly $d_{\Si_j'}(x,x_2)\leq m^{-i}$.
Therefore (\ref{eqn-dsij'}) holds.

In general, construct a new admissible inverse system $\{Y_k\}$ satisfying Assumption
\ref{ass-tempfinite} by letting $Y_k$ be the disjoint union of $X_k$ with a copy of $\R$
when $i<k\leq j$, and $Y_k=X_k$ otherwise.  Then extend the projection map
$\pi_i:X_{i+1}\ra X_i$ to $\pi_i:Y_{i+1}\ra Y_i=X_i$ by mapping 
$Y_{i+1}\setminus X_{i+1}\simeq \R$ to a monotone geodesic containing $v$.   Then for 
$i<k<j$ extend 
$\pi_j:X_{j+1}\ra X_j$ to $\pi_j:Y_{j+1}\ra Y_j$ by mapping $Y_{j+1}\setminus X_{j+1}\simeq \R$
isomorphically to $Y_j\setminus X_j\simeq\R$.   Then there is a system of measures
$\{\Si_{k,Y}'\}$ for the inverse system $\{Y_k\}$, and it follows that the associated 
cut metric $d_{\Si_j'}^Y$ is the same for pairs  $x_1,x_2\subset X_j\subset Y_j$.  
Since $v$ belongs
to the image of $\pi_i^j:Y_j\ra Y_i$, we have $
d_{\Si_j'}^X(x_1,x_2)=d_{\Si_j'}^Y(x_1,x_2)\leq 2m^{-i}$.
\end{proof}

\section{Proof of Theorem \ref{thm-bilipschitzembedding} under 
Assumption \ref{ass-tempfinite}}
\label{sec-completion_assumption}

We will define a sequence $\{\rho_i:X_\infty\times X_\infty\ra [0,\infty)\}$ of pseudo-distances
on $X_\infty$, such that $\rho_i$ is induced by a map $X_\infty\ra L_1$,
and $\rho_i$ converges uniformly to some $\rho_\infty:X_\infty\times X_\infty\ra [0,\infty)$.  
By a standard argument, this yields an isometric embedding $(X_\infty,\rho_\infty)\ra L_1$.
(The theory of ultralimits \cite{heinmank,benlin} implies the metric space $(X_\infty,\rho_\infty)$
isometrically embeds in an ultralimit $V$ of $L_1$ spaces;   by Kakutani's theorem
\cite{kakutani} the space $V$ is isometric to  an $L_1$ space, and so $\rho_\infty$ isometrically
embeds in $L_1$.)
To complete the proof,
 it suffices to verify that $\rho_\infty$ has the properties asserted
by the theorem.  (Alternately, one may construct a cut measure $\Si_\infty$ on $X_\infty$ as weak
limit, and use the corresponding cut metric to provide the embedding to $L_1$.)

Let $\rho_i=(\pi_i^\infty)^*d_{\Si_i'}$ be the pullback of $d_{\Si_i'}$ to $X_\infty$. 
By Lemma \ref{lem-distancechange} we have $|\rho_{i+1}-\rho_i|\leq 4\,m^{-(i+1)}$, so
the sequence $\{\rho_i\}$ converges uniformly to a pseudo-distance 
$\rho_\infty:X_\infty\times X_\infty\ra[0,\infty)$.   

\begin{lemma}

\mbox{}
\begin{enumerate}
\setlength{\itemsep}{.5ex}
\item $\rho_\infty(x,x')\leq \bar d_\infty(x,x')$.
\item $
\rho_\infty(x,x')\geq \frac{A}{2m^3}\cdot\bar d_\infty(x,x')\,.
$
\item $\rho_\infty(x,x')=\bar d_\infty(x,x')$ if $x,x'$ lie on a monotone geodesic.
\end{enumerate}
\end{lemma}
\begin{proof}
(1).  Suppose $x,x'\in X_\infty$, and for some $i\in \Z$ the projections 
$\{\pi_i(x),\pi_i(x')\}$ are contained in $ \Star(v,X_i)$.
By Corollary \ref{cor-dsistar} we have
$\rho_\infty(x,x')\leq 2m^{-i}$.   Now (1) follows from the definition
of $\bar d_\infty$.

(2).  Suppose
$x\neq x'$, and let $j\in \Z$ be the minimum of the indices $k\in \Z$ such that
$\pi_k(x),\pi_k(x')$ are not contained in the trimmed star of any vertex $v\in X_k$.
Then $\bar d_\infty(x,x')\leq 2m^{-(j-1)}$ by Lemma \ref{lem-trimmedstarscale}, while
$\rho_\infty(x,x')\geq Am^{-(j+2)}$ by Lemma \ref{lem-notinsametrimmedstar}.   Thus
$$
\rho_\infty(x,x')\geq \frac{A}{2m^3}\cdot\bar d_\infty(x,x')\,.
$$

(3).  If $x,x'\in X_\infty$ lie on a monotone geodesic $\ga$, then 
$\ga$ will  project homeomorphically under $\pi_i$ to a monotone geodesic
$\pi_i(\ga)$, which contains at least $\bar d_\infty(x_1,x_2)m^{(i+1)}-2$ vertices of $V_i'$.
By Lemma \ref{lem-normalized} we have $\rho_i(x_1,x_2)\geq \bar d_\infty(x_1,x_2)-2m^{-(i+1)}$.
Since $i$ was arbitrary we get $\rho_\infty(x,x')=\bar d_\infty(x,x')$.
\end{proof}

This completes the proof 
of Theorem \ref{thm-bilipschitzembedding} under 
Assumption \ref{ass-tempfinite}.

\section{The proof of Theorem \ref{thm-bilipschitzembedding}, general case}
\label{sec-generalcase}

We recall the three conditions from Assumption \ref{ass-tempfinite}:
\begin{enumerate}
\setlength{\itemsep}{.5ex}
\item $\pi_i$ is finite-to-one for all $i\in\Z$.
\item $X_i\simeq\R$ and $\pi_i:X_{i+1}\ra X_i'$ is an isomorphism for all $i\leq 0$.
\item For every $i\in\Z$,  every vertex $v\in V_i$ has neighbors $v_\pm\in V_i$
with $v_-\prec v\prec v_+$.   Equivalently, $X_i$ is a union of  monotone
geodesics (see Definition \ref{def-monotone_geodesic}). 
\end{enumerate}

In this section these three conditions will be removed in turn.

\subsection{Removing the finiteness assumption}
\label{subsec-(2)and(3)}
We now assume that $\{X_i\}$ is an admissible inverse system satisfying  conditions (2) and (3) of
Assumption \ref{ass-tempfinite}, but not necessarily the finiteness condition (1).

To prove Theorem \ref{thm-bilipschitzembedding} without the finiteness assumption,
 we observe that the construction 
of the distance function
$\rho_\infty:X_\infty\times X_\infty\ra \R$ can be reduced to the case already  treated, in the sense that
for any two points $x_1,x_2\in X_\infty$, we can apply the construction of the cut metrics to 
finite valence subsystems, and the resulting distance $\rho_\infty(x_1,x_2)$ is independent of the
choice of subsystem.  The proof is then completed by invoking the main result of 
\cite{krivineetal}.   We now give the
details.

\begin{definition}
A {\bf finite subsystem} of the inverse system $\{X_j\}$ is a  collection of subcomplexes
$\{Y_j\subset X_j\}_{j=-\infty}^i$, for some $i\geq 0$, such that $\pi_j(Y_{j+1})\subset Y_j$
for all $j<i$, and $\{Y_j\}_{j\leq i}$ satisfies Assumption \ref{ass-tempfinite} for indices $\leq i$.
In other words:
\begin{enumerate}
\setlength{\itemsep}{.5ex}
\item $\pi_j$ is finite-to-one for all $j\leq i$.
\item $Y_j\simeq\R$ and $\pi_j:Y_{j+1}\ra Y_j'$ is an isomorphism for all $j\leq 0$.
\item $Y_j$  is a union of (complete) monotone
geodesics for all $j\leq i$.
\end{enumerate}
\end{definition}

Suppose $i\geq 0$ and $V$ is a finite subset of $V_i'\subset X_i$.
Then  there exists a finite subsystem 
$\{Y_j\}_{j=-\infty}^i$ such that
$Y_i$ contains
$V$.   One  may be obtain such a system by letting
$Y_i$ be  a finite union of monotone geodesics in $X_i$ which contains $V$, and 
taking $Y_j=\pi_j^i(Y_i)$ for 
$j<i$.
The inductive construction of the slice measures  in the finite valence case
may be applied to the finite subsystem $\{Y_j\}$, to obtain a sequence of slice measures 
which we denote
by $\Si_{j,Y}'$, to emphasize the potential dependence on $Y$.

\begin{lemma}
\label{lem-sii'independent}
Suppose $i\geq 0$, $V\subset V_i'$ is a finite subset, and let
$\{Y_j\}_{j\leq i_Y}$ and $\{Z_j\}_{j\leq i_Z}$ be  finite subsystems of $\{X_j\}$,
where $i\leq \min(i_Y,i_Z)$ and $V\subset Y_i\cap Z_i$.
If $\Si_{i,Y}'$, $\Si_{i,Z}'$ denote the respective slice measures, then 
$$
\Si_{i,Y}'(\{S\in \slice_{i,Y}'\mid S\supset V\})=\Si_{i,Z}'(\{S\in \slice_{i,Z}'\mid S\supset V\})\,,
$$
 i.e. 
the slice measure does not depend on the choice of subsystem
containing $V$.
\end{lemma}
\begin{proof}
If $i\leq 0$ then $Y_i=Z_i$ and $\Si_{i,Y}'=\Si_{i,Z}'$ by construction.  So assume that $i>0$, and that
the lemma holds for all finite subsets of $V_{\bar i}'$ for all $\bar i<i$.

Suppose $S'\in \slice_{i,Y}'$ is a child of $S\in \slice_{i-1,Y}'$, and $V\subset S'$.   Then for every 
$v\in V$, either $v\in V_i$, in which case $\pi_{i-1}(v)\in S$, or $v\in V_i'\setminus V_i$, 
in which case $v$ is an interior point of some edge $\ol{u_1u_2}$ of $Y_i\subset X_i$, and $S$ contains precisely
one of the points $\pi_{i-1}(u_1)$, $\pi_{i-1}(u_2)$.   By the definition   of $\Si_{i,Y}'$ given 
by  (\ref{eqn-diffusion}):
\begin{align}
\label{eqn-slicecontainingv}
\Si_{i,Y}'(\{S'&\in \slice_{i,Y}'\mid  S'\supset V\})\nonumber\\
=&\sum_{S\in\slice_{i-1,Y}'}\,K_{S}(\{S'\in \slice_{i,Y}'\mid S'\supset V\})\,\Si_{i-1,Y}'(S)
\end{align}
By the above observation, the nonzero terms in  the sum come from  the slices $S\in \slice_{i-1,Y}'$
which contain  precisely one of a finite collection of finite subsets $\bar V_1,\ldots, \bar V_k\subset V_{i-1}'$.
If $S\in \slice_{i-1,Y}'$ contains $\bar V_l$, 
then from  the definition of $K_S$, the quantity $K_{S}(\{S'\in \slice_i'\mid S'\supset V\})$ depends only
on $V_l$.   Therefore by the induction assumption, it follows that the nonzero terms in (\ref{eqn-slicecontainingv})
will be the same as the corresponding terms in the sum defining $\Si_{i,Z}(\{S'\in \slice_{i,Y}'\mid S'\supset V\})$. 

\end{proof}

\begin{lemma}
\label{lem-cutmetricindependentofy}
If $\{Y_j\}_{j\leq i}$ is a finite subsystem such that $Y_i$ contains
$\{x_1,x_2\}\subset X_i$, then the cut metric $d_{\Si_{i,Y}'}(x_1,x_2)$ does not depend on the choice
of $\{Y_j\}_{j\leq i}$.
\end{lemma}
\begin{proof}
Let $\ga_1,\ga_2\subset Y_i$ be monotone geodesics containing $x_1$ and $x_2$ respectively.
Then $d_{\Si_{i,Y}'}(x_1,x_2)$ is the total $\Si_{i,Y}'$-measure of the slices $S\in \slice_{i,Y}'$ such that
either $x_1\preceq S \prec x_2$ or $x_2\preceq S\prec x_1$.   But  every such slice $S$ contains
precisely one point from $\ga_1$, and one point from $\ga_2$.  
As the choice of $\ga_1$, $\ga_2$ was arbitrary,  Lemma \ref{lem-sii'independent} 
implies that cut metric $d_{\Si_{i,Y}'}(x_1,x_2)$ does not change when we pass from $\{Y_j\}_{j\leq i}$ to 
another subsystem which contains $\{Y_j\}_{j\leq i}$.    This implies the lemma, since the union of any two
finite subsystems containing $\{x_1,x_2\}$ is a finite subsystem  which assigns the same cut metric
to $(x_1,x_2)$. 
\end{proof}

We now define a sequence of pseudo-distances $\{\rho_i:X_\infty\times X_\infty\ra [0,\infty)$
by letting $\rho_i(x_1,x_2)=d_{\Si_{i,Y}}(\pi_i^\infty(x_1),\pi_i^\infty(x_2))$
where $\{Y_j\}_{j\leq i}$ is any finite subsystem containing $\{\pi_i^\infty(x_1),\pi_i^\infty(x_2)\}$. 
By Lemma \ref{lem-cutmetricindependentofy} the pseudo-distance is well-defined.  As in the finite
valence case:
\begin{itemize}
\item Lemma \ref{lem-distancechange} implies that $\{\rho_i\}$ converges uniformly to a
pseudo-distance $\rho_\infty:X_\infty\times X_\infty\ra [0,\infty)$.  
 \item $\frac{A}{2m^3}\bar d_\infty\leq\rho_\infty\leq \bar d_\infty$, since this may be verified for
 each pair of points $x_1,x_2\in X_\infty$ at a time, by using finite subsystems.
\item  If $V\subset X_\infty$
is a finite subset, then the restriction of $\rho_i$  to $V$ embeds isometrically in $L_1$ for all $i$, and
hence the same is true for $\rho_\infty$. 
\end{itemize}
By the main result of \cite{krivineetal},  if $Z$ is a metric space such that every finite subset isometrically embeds
in $L_1$, then $Z$ itself isometrically embeds in $L_1$.  Therefore
$(X_\infty,\rho_\infty)$ isometrically embeds in $L_1$.

\subsection{Removing Assumption \ref{ass-tempfinite}(2)}
\label{subsec-unionmonotone}
Now suppose $\{X_i\}$ is an admissible inverse system satisfying Assumption \ref{ass-tempfinite}(3),
i.e. it is a union of monotone geodesics.   We will reduce to the case treated
in Section \ref{subsec-(2)and(3)} by working with balls, and then take an ultralimit
as the radius tends to infinity.

\begin{lemma}
\label{lem-embedballinsystem}
Suppose $p\in X_\infty$, $R\in (0,\infty)$ and $R<m^{-(i+1)}$.   Then there is 
an admissible inverse system $\{Z_j\}$ satisfying (2) and (3) of Assumption
\ref{ass-tempfinite}, and an isometric embedding of the rescaled ball:
$$
\phi:(B(p,R),m^{(i-1)}\bar d_\infty)\ra Z_\infty
$$
which preserves the partial order, i.e. if $x,y\in B(p,R)$ and $x\preceq y$, then 
$\phi(x)\preceq \phi(y)$.
\end{lemma}
\begin{proof}
Since $R<m^{-(i+1)}$, by Lemma \ref{lem-starscale}
there is a  $v\in V_i$ such that $\pi_i(B(p,R))\subset \tStar(v,X_i)$.

We now construct an inverse system $\{Y_j\}$ as follows.  For $j \geq i$, we
let $Y_j$ be the inverse image of $\Star(v,X_i)$ under the projection
$\pi_i^j:X_j\ra X_i$.   We let $Y_j\simeq \R$ for $j< i$.   To define the projection
maps, we take $\pi_j^Y=\pi_i^X\restr_{Y_{j+1}}$ for $j\geq i$, and let $\pi_j^Y:Y_{j+1}\ra Y_j'$
be a simplicial isomorphism for $j<i-1$.  Finally,  we take
$\pi_{i-1}^Y:Y_i=\Star(v,X_j)\ra Y_{i-1}'\simeq \R$ to be
an order preserving simplicial map which is an isomorphism on edges, thus the star
$Y_i=\Star(v,X_i)$ is collapsed onto two consecutive edges $\ol{w_-w}$, 
$\ol{ww_+}$  in $Y_{i-1}'$, where $w=\pi_{i-1}(v)$, $w_-\prec w$, and
$w\prec w_+$.   Thus $\{Y_j\}$ is an admissible inverse system, but it
need not satisfy (2) or (3) of Assumption \ref{ass-tempfinite}.

Next, we enlarge $\{Y_j\}$ to a system $\{\hat Y_j\}$. 
We first attach, for every $k\geq i$, and every
vertex $z\in (\pi_{i-1}^k)^{-1}(w_-) $, a directed ray $\ga_z$
which is directed isomorphic
to $(-\infty,0]$ with the usual subdivision and order.   We then  extend the projection maps
so that if $\pi_k(z)=z'$ then $\ga_z\subset Y_{j+1}$ is mapped direction-preserving
isomorphically to a ray in
$X_j'$ starting at $z'$.  Then similarly, we attach directed rays to vertices 
$z\in (\pi_{i-1}^k)^{-1}(w_+)$, and extend the projection maps.  

Finally, we let $\{Z_j\}$ be the system obtained from $\{\hat Y_j\}$ by shifting 
indices by $(i-1)$, in other words  $Z_j=Y_{j-i-1}$.    

 Then $\{Z_j\}$ satisfies (2) and (3) of 
Assumption \ref{ass-tempfinite}.    For all $j\geq i$,
we have compatible  direction preserving simpliicial
embeddings $X_j\supset (\pi_i^j)^{-1}(\Star(v,X_i))\ra Z_{j-i+1}$.  
We will identify points with their image under this embedding.
If $x,x'\in B(p,R)$ and
$x=x_0,\ldots,x_k=x'$ is a chain of points as in
 Lemma \ref{lem-altdefdbarinfty} which nearly realizes $\bar d_\infty^X(x,x')$, then the chain
and the associated stars will project into $\Star(v,X_i)$; this implies that 
$\bar d_\infty^Z(x,x')\leq m^{(i-1)}\bar d_\infty^X(x,x')$.
Similar reasoning gives $m^{(i-1)}\bar d_\infty^X(x,x')
\leq \bar d_\infty^Z(x,x')$.

If $\ga\subset B(p,R)$
is a monotone geodesic segment, then $\pi_i(\ga)$ is a monotone geodesic segment in 
$\Star(v,X_i)$ with endpoints in $\Star(v,X_i)$, and so $\pi_i(\ga)\subset\Star(v,X_i)$.
Thus the embedding also preserves the partial order as claimed.

\end{proof}

\bigskip
Fix $p\in X_\infty$.  Then for every $n\in \N$, since $m^n<m^{(n+1)}$, 
Lemma \ref{lem-embedballinsystem} provides an inverse system
$\{Z_j^n\}_{j\in \Z}$ and an embedding
$$
\phi_n:(B(p,m^n),m^{ -n-3}\bar d_\infty)\ra Z_\infty^n\,.
$$
Let $f_n:Z_\infty^n\ra L_1$ be a $1$-Lipschitz embedding satisfying 
the conclusion of Theorem \ref{thm-bilipschitzembedding}, constructed in Section 
\ref{subsec-(2)and(3)}, and let $\psi_n:(B(p,m^n),\bar d_\infty)\ra L_1$ be the
composition $f_n\circ \phi_n$, rescaled by $m^{n+3}$.  Next we use a standard
argument with ultralimits, see \cite{benlin}.  Then the 
ultralimit  
$$
\ulim\psi_n:\ulim (B(p,m^n),\bar d_\infty)\ra\ulim L_1
$$
 yields 
the desired $1$-Lipschitz embedding, since $X_\infty$ embeds canonically
and isometrically in $\ulim (B(p,m^n),\bar d_\infty)$, and
an ultralimit of a sequence of $L_1$ spaces
is an $L_1$ space \cite{kakutani}.

\subsection{Removing Assumption \ref{ass-tempfinite}(3)}

Let $\{X_i\}$ be an admissible inverse system.

\begin{lemma}
\label{lem-hatx_i}
 $\{X_i\}$ may be enlarged  to an admissible inverse system $\{\hat X_i\}$ such that
for all $i\in\Z$,  $\hat X_i$ is a union of monotone geodesics.
\end{lemma}
\begin{proof}
We first enlarge  $X_i$ to $\hat X_i$ as follows.  For each $i\in \Z$, and each $v\in V_i$ which
does not have a neighbor $w\in V_i$ with $w\prec v$ (respectively $v\prec w$), 
we attach a directed ray $\ga_v^-$  (respectively $\ga_v^+$)
which is directed isomorphic
to $(-\infty,0]$ (respectively $[0,\infty)$) 
with the usual subdivision and order.     The resulting graphs $\hat X_i$ have the 
property that every vertex  $v\in\hat X_i'$ is the initial vertex of directed rays in both
directions.  Therefore we may extend the projection maps $\pi_i:X_{i+1}\ra X_i$
by mapping $\ga_v^\pm\subset X_{i+1}$ direction-preserving  isomorphically to a 
ray starting at $\pi_i(v)\in X_i'$.  The resulting inverse system is admissible.
\end{proof}

If $\bar d_\infty^X$ and $\bar d_\infty^{\hat X}$ are the respective metrics,
then for all $x,x'\in X_\infty\subset\hat X_\infty$, we clearly have 
$\bar d_\infty^{\hat X}(x,x')\leq \bar d_\infty^X(x,x')$.   Note that  if $x,x'\in X_\infty$ and
$\{\pi_j(x),\pi_j(x')\}$ belong to the trimmed star of a vertex $v\in \hat X_j$, then
in fact $v$ is a vertex of $X_j$ (since the trimmed star of a vertex in $\hat X_j\setminus X_j$
does not intersect $X_j$).   Thus by Lemma \ref{lem-trimmedstarscale}
we have $\bar d_\infty^{\hat X}(x,x')\geq \frac{2m^2}{(m-2)}\bar d_\infty^X(x,x')$.
Therefore if $f:\hat X_\infty\ra L_1$ is the embedding given by Section \ref{subsec-unionmonotone},
then the composition $X_\infty\hookrightarrow \hat X_\infty\stackrel{f}{\ra}L_1$ 
satisfies the requirements of Theorem \ref{thm-bilipschitzembedding}.

\section{The Laakso examples from \cite{laakso} and Example \ref{ex-laaksoahlforsregularpi}}
\label{sec-laaksoexamples}

In \cite{laakso} Laakso constructed Ahlfors $Q$-regular metric spaces satisfying a Poincare 
inequality for all $Q>1$.  
In the section we will show that the simplest example from \cite{laakso} is isometric to Example
\ref{ex-laaksoahlforsregularpi}.

\subsection{Laakso's description}
We will (more or less) follow Section 1 of \cite{laakso}, in the special case that (in Laakso's notation)
the Hausdorff dimension
$Q=1+\frac{\log 2}{\log 3}$, $t=\frac13$, and $K\subset [0,1]$ is the middle third Cantor set.

Define $\phi_0:K\ra K$, $\phi_1:K\ra K$ by 
$$
\phi_0(x)=\frac13 x\,,\quad \phi_1(x)=\frac23 +\frac13 x\,.
$$
Then $\phi_0$ and $\phi_1$ generate a semigroup of self-maps $K\ra K$.
Given a binary string $a=(a_1,\ldots,a_k)\in \{0,1\}^k$, we let $|a|=k$ denote its length. 
For every $a$, let $K_a\subset K$, be the image of $K$ under the corresponding word in the
the semigroup:
$$
K_a=\phi_{a_1}\circ\ldots\circ\phi_{a_k}(K)\,.
$$
Thus for every $k\in \N$ we have a decomposition of $K$ into a disjoint union 
$K=\sqcup_{|a|=k}\,K_a$.

For each $k\in \N$, let $S_k\subset [0,1]$ denote the set of $x\in [0,1]$ with a finite
ternary expansion $x=.m_1\ldots m_k$ where the last digit $m_k$ is nonzero.   In other
words, if $V_j$ is the set of vertices of the subdivision of $[0,1]$
into  intervals of length $3^{-j}$  for $j\geq 0$, then $S_k=V_k\setminus V_{k-1}$. 

For each $k\in\N$ we define an equivalence relation $\sim_k$ on $[0,1]\times K$ as follows.
For every $q\in S_k$, and every binary string $a=(a_1,\ldots,a_k)$, we identify
$\{q\}\times K_{(a_1,\ldots,a_k,0)}$ with $\{q\}\times K_{(a_1,\ldots,a_k,1)}$ by translation, 
or equivalently, for all $x\in K$, we identify
$
\phi_{a_1}\circ\ldots\circ \phi_{a_k}\circ\phi_0(x)$ and
$\phi_{a_1}\circ\ldots\circ \phi_{a_k}\circ\phi_1(x)$.    

Let $\sim$ be the union of the equivalence relations $\{\sim_k\}_{k\in \N}$;  this is an
equivalence relation.   We denote the collection of cosets  $([0,1]\times K)/\sim$ by $F$,  equip
it with the quotient topology, and let $\pi:[0,1]\times K\ra F$ be the canonical surjection.
The distance function on $F$ is defined by 
$$
d(x,x')=\inf\{\h^1(\ga)\mid \ga\subset [0,1]\times K,\;
\pi(\ga)\;\text{contains a path from $x$ to $x'$}\}\,,
$$
where $\h^1$ denotes $1$-dimensional Hausdorff measure.

\subsection{Comparing $F$ with Example \ref{ex-laaksoahlforsregularpi}}
For every $k\in \N$ we will construct $1$-Lipschitz maps $\iota_k:X_k\ra F$, $f_k:F\ra X_k$
such that $f_k\circ\iota_k=\id_{X_k}$, such that the image of $\iota_k$ is $\const\cdot 3^{-k}$-dense
in $F$.   This implies that $\iota_k$ is an isometric embedding for all $k$, and is
a $\const\cdot 3^{-k}$-Gromov-Hausdorff approximation.   Therefore $F$
is the Gromov-Hausdorff limit of the sequence $\{X_k\}$, and is isometric to $(X_\infty,d_\infty)$.

For every $k$, there is a quotient map $K\ra \{0,1\}^k$ which maps the subset
$K_{(a_1,\ldots,a_k)}\subset K$ to $(a_1,\ldots,a_k)$.  This induces quotient
maps $K\times [0,1]\ra \{0,1\}^k\times [0,1]$, and $f_k:F\ra X_k$,
where $X_k$ is the graph from Example
\ref{ex-laaksoahlforsregularpi}.    When $X_k$ is equipped with the 
path metric described in the example, the map $f_k$ is $1$-Lipschitz, because
any set $U\subset [0,1]\times K$ with diameter $<3^{-k}$
projects under the composition
$[0,1]\times K\ra F\stackrel{f_k}{\ra}X_k$ to a set $\bar U\subset X_k$
with $\diam(\bar U)\leq \diam(U)$.  

For every $k$, there is an injective map $\{0,1\}^k\ra K$ which sends
$(a_1,\ldots,a_k)$ to the smallest element of $K_a$, i.e.
$\phi_{a_1}\circ\ldots\circ\phi_{a_k}(0)$.   This induces
maps $[0,1]\times \{0,1\}^k\ra [0,1]\times K$ and $\iota_k:X_k\ra F$.
It follows from the definition of the metric on $F$ that $\iota_k$ is
$1$-Lipschitz, since geodesics in $X_k$ can be lifted piecewise isometrically
to segments in $[0,1]\times K$.  

We have $f_k\circ \iota_k=\id_{X_k}$.  Therefore $\iota_k$ is an isometric
embedding.  Given $x\in [0,1]\times K$, there exist $i\in \{0,\ldots,3^k\}$,
$a\in \{0,1\}^k$ such that 
$x\in W=[\frac{i-1}{3^{k}},\frac{i}{3^{k}}]\times K_a$.    Now $W/\sim$ is a subset of $F$ 
which intersects $\iota_k(X_k)$, and which has
  diameter $\leq 3^{-k}\diam(F)$ due to the self-similarity of the equivalence relation,
so $\iota_k$ is a $3^{-k}\diam(F)$-Gromov-Hausdorff approximation.

\section{Realizing metric spaces as inverse limits: further generalization}
\label{sec-realization_generalization}
In this section we consider the realization problem in greater generality.

Let $f:Z\ra Y$ be a $1$-Lipschitz map between metric spaces.
We  assume that for all $r\in (0,\infty)$,
if  $U\subset Y$ and $\diam(U)\leq r$, then the $r$-components
of  $f^{-1}(U)$ have diameter at most $Cr$.

\begin{remark}
Some variants of this assumption are essentially
equivalent.  
Suppose $C_1,C_2,\bar C_1\in (0,\infty)$.  If for all $r\in (0,\infty)$ and 
every subset $U\subset Y$ with $\diam(U)\leq r$, the $C_1r$-components
of $f^{-1}(U)$ have diameter $\leq C_2r$, it follows easily that the $\bar C_1r$-components
of $f^{-1}(U)$ have diameter $\leq C_2r\cdot\max(1,\frac{\bar C_1}{C_1})$.
\end{remark}

\bigskip

\subsection{Realization as an inverse limit of simplicial complexes}
Fix $m\in (1,\infty)$ and  $A\in (0,1)$.  For 
every $i\in \Z$, let $\U_i$ be an open cover of $Y$ such that for
all $i\in \Z$:
\begin{enumerate}
\setlength{\itemsep}{.5ex}
\item The cover $\U_{i+1}$ is a refinement of $\U_i$.
\item Every $U\in \U_i$ has diameter $\leq m^{-i}$. 
\item For every $y\in Y$, the  ball $B(y,Am^{-i})$
is contained in some $U\in\U_i$.
\end{enumerate}

Next, for all $i\in \Z$ we let $f^{-1}(\U_i)=\{f^{-1}(U)\mid U\in\U_i\}$,
and define $\hat \U_i$ to be the collection of pairs $(\hat U,U)$
where $U\in \U_i$ and $\hat U$ is an $m^{-i}$-component  of $f^{-1}(U)$.   

We obtain  inverse systems of simplicial complexes $\{L_i=\nerve(\U_i)\}_{i\in \Z}$, and 
 $\{K_i=\nerve(\hat\U_i)\}_{i\in\Z}$, where we view $\hat \U_i$ as an open cover of
$Z$ indexed by the elements of $\hat\U_i$.
There are canonical simplicial maps $K_i\ra L_i$ which send $(\hat U,U)\in \hat\U_i$
to $U\in \U_i$.

We may define a metric $d_{K_\infty}$
on the inverse limit $K_\infty$ by taking the supremal metric on $K_\infty$ such that
for all $i\in \Z$ and every vertex $v\in K_i$, the inverse image of the closed star $\Star(v,K_i)$
under the projection $K_\infty\ra K_i$ has diameter $\leq m^{-i}$.
Let $\bar K_\infty$ be the completion of $(K_\infty,d_{K_\infty})$.  

For every $z\in Z$ and  $i\in \Z$, there is a canonical (possibly infinite dimensional)
simplex $\si_i$ in $K_i$ corresponding to the collection of $U\in \hat \U_i$ which contain $z$.
The inverse images $(\pi_i^\infty)^{-1}(\si_i)\subset \bar K_\infty$ form a nested sequence of
subsets with diameter tending to zero, so they determine a unique point in the complete
space $\bar K_\infty$.   This defines a map $\phi:Z\ra \bar K_\infty$.

\begin{proposition}
$\phi$ is a bilipschitz homeomorphism.
\end{proposition}
\begin{proof}
If $z,z'\in Z$ and $d(z,z')\leq Am^{-i}$, then $f(z),f(z')\in U$ for some $U\in \U_i$, and hence
$z,z'\in \hat U$ for some $m^{-i}$-component $\hat U\in \hat\U_i$ of $U$.  It follows that 
$d_{K_\infty}(\phi(z),\phi(z'))\leq m^{-i}$.

If $z,z'\in Z$ and $d_{K_\infty}(\phi(z),\phi(z'))\leq m^{-i}$, it follows from the definitions that
$d(z,z')\lesssim m^{-i}$.

\end{proof}

\bigskip

There is  another metric $\bar d_\infty$ on $Z$, namely the supremal metric
with the property that every element of $\hat \U_i$ has diameter at most $m^{-i}$.
Reasoning similar to the above shows that $\bar d_\infty$ is comparable to $d_Z$.

\bigskip
\subsection{Factoring $f$ into ``locally injective'' maps}

Let $\{\U_i\}_{i\in\Z}$ be a sequence of open covers as above.

For every $i\in \Z$, we may define a relation on $Z$ by declaring
that $z,z'\in Z$ are related if $f(z)=f(z')$ and $\{z,z'\}\subset \hat U$
for  some $(\hat U,U)\in \hat \U_i$.    We let $\sim_i$ be the equivalence relation this generates.
Note that $\sim_{i+1}$ is a finer equivalence relation than $\sim_i$.

For every $i\in \Z$, we have a pseudo-distance $d_i$ on $Z$ defined by 
letting $d_i$ be the supremal distance function $\leq d_Z$ such that $d_i(z,z')=0$
whenever $z\sim_i z'$.   Then $d_i\leq d_{i+1}\leq d_Z$, so we have a well-defined limiting
distance function $d_\infty:Z\times Z\ra [0,\infty)$.
We let $Z_i$ be the metric space obtained from $(Z,d_i)$ by collapsing zero diameter
subsets to points.   We get an inverse system $\{Z_i\}_{i\in\Z}$ with $1$-Lipschitz 
projection maps, and a compatible family of mappings $f_i:Z_i\ra Y$ induced by $f$. 

The map $f_i$ is ``injective at scale $\simeq m^{-i}$'' in the following sense.  If $z\in Z$,
and $\bar B\subset Z_i$ is the image of the ball $B(z,Am^{-i})$ under the canonical projection
map $Z\ra Z_i$, then the restriction of $f_i$ to $\bar B$ is injective.

\begin{proposition}
If $z,z'\in Z$ and $d_i(z,z')< m^{-i}$, then $d(z,z')\lesssim m^{-i}$.  Consequently
$d_\infty\simeq d_Z$.
\end{proposition}
\begin{proof}
If $z_1,z_2\in Z$ and $z_1\sim_i z_2$, then $z_1,z_2$ belong to the same $m^{-i}$-component of
$f^{-1}(\ol{B(f(z_1),2m^{-i})}))$, and hence $d(z_1,z_2)\leq 2Cm^{-i}$.

If $z,z'\in Z$ and $d_i(z,z')< m^{-i}$, then there are points $z=z_0,\ldots,z_k=z'\in Z$
such if 
$$
J=\{j\in \{1,\ldots,k\}\mid z_{j-1}\not\sim_i z_j\}
$$
then 
$$
\sum_{j\in J} \;d(z_{j-1},z_j)<m^{-i}\,.
$$
Since $f$ is $1$-Lipschitz, it follows that $f(z_j)\in B(f(z),m^{-i})$ for all $j\in \{1,\ldots,k\}$.
Moreover, the $z_j$'s lie in the same $2m^{-i}$-component of $f^{-1}(B(f(z),m^{-i}))$, so
$d(z,z')\leq 2Cm^{-i}$.

\end{proof}

\bibliography{laaksoplus}

\def\cprime{$'$}
\begin{thebibliography}{DCK72}

\bibitem[Ass80]{assouad}
Patrice Assouad.
\newblock Plongements isom\'etriques dans {$L\sp{1}$}: aspect analytique.
\newblock In {\em Initiation Seminar on Analysis: G. Choquet-M. Rogalski-J.
  Saint-Raymond, 19th Year: 1979/1980}, volume~41 of {\em Publ. Math. Univ.
  Pierre et Marie Curie}, pages Exp. No. 14, 23. Univ. Paris VI, Paris, 1980.

\bibitem[BL00]{benlin}
Y.~Benyamini and J.~Lindenstrauss.
\newblock {\em Geometric nonlinear functional analysis. {V}ol. 1}, volume~48 of
  {\em American Mathematical Society Colloquium Publications}.
\newblock American Mathematical Society, Providence, RI, 2000.

\bibitem[Che99]{cheeger}
J.~Cheeger.
\newblock Differentiability of {L}ipschitz functions on metric measure spaces.
\newblock {\em Geom. Funct. Anal.}, 9(3):428--517, 1999.

\bibitem[CKa]{ck-piexamples}
J.~Cheeger and B.~Kleiner.
\newblock in preparation.

\bibitem[CKb]{ckbv-old}
J.~Cheeger and B.~Kleiner.
\newblock Differentiating maps to ${L}^1$ and the geometry of {B}{V} functions.
\newblock \url{http://arxiv.org/abs/math/0611954}.

\bibitem[CKc]{ckmetdiff}
J.~Cheeger and B.~Kleiner.
\newblock Metric differentiation for {P}{I} spaces.
\newblock In preparation.

\bibitem[CK06a]{GFDA}
J.~Cheeger and B~Kleiner.
\newblock On the differentiability of {L}ipschtz maps from metric measure
  spaces into banach spaces.
\newblock In {\em Inspired by S.S. Chern, A Memorial volume in honor of a great
  mathematician}, volume~11 of {\em Nankai tracts in Mathematics}, pages
  129--152. World Scientific, Singapore, 2006.

\bibitem[CK06b]{crannouncement}
Jeff Cheeger and Bruce Kleiner.
\newblock Generalized differential and bi-{L}ipschitz nonembedding in {$L^1$}.
\newblock {\em C. R. Math. Acad. Sci. Paris}, 343(5):297--301, 2006.

\bibitem[CK09]{ckdppi}
Jeff Cheeger and Bruce Kleiner.
\newblock Differentiability of {L}ipschitz maps from metric measure spaces to
  {B}anach spaces with the {R}adon-{N}ikod\'ym property.
\newblock {\em Geom. Funct. Anal.}, 19(4):1017--1028, 2009.

\bibitem[CK10a]{ckbv}
Jeff Cheeger and Bruce Kleiner.
\newblock Differentiating maps into {$L^1$}, and the geometry of {BV}
  functions.
\newblock {\em Ann. of Math. (2)}, 171(2):1347--1385, 2010.

\bibitem[CK10b]{ckmetmon}
Jeff Cheeger and Bruce Kleiner.
\newblock Metric differentiation, monotonicity and maps to {$L^1$}.
\newblock {\em Invent. Math.}, 182(2):335--370, 2010.

\bibitem[CKN09]{CKN09}
Jeff Cheeger, Bruce Kleiner, and Assaf Naor.
\newblock A {$(\log n)^{\Omega(1)}$} integrality gap for the sparsest cut
  {SDP}.
\newblock In {\em 2009 50th {A}nnual {IEEE} {S}ymposium on {F}oundations of
  {C}omputer {S}cience ({FOCS} 2009)}, pages 555--564. IEEE Computer Soc., Los
  Alamitos, CA, 2009.

\bibitem[DCK72]{krivineetal}
D.~Dacunha-Castelle and J.~L. Krivine.
\newblock Applications des ultraproduits \`a l'\'etude des espaces et des
  alg\`ebres de {B}anach.
\newblock {\em Studia Math.}, 41:315--334, 1972.

\bibitem[DL97]{dezalaur}
M.~Deza and M.~Laurent.
\newblock {\em Geometry of cuts and metrics}, volume~15 of {\em Algorithms and
  Combinatorics}.
\newblock Springer-Verlag, Berlin, 1997.

\bibitem[DU97]{dranishnikov}
A.~N. Dranishnikov and V.~V. Uspenskij.
\newblock Light maps and extensional dimension.
\newblock {\em Topology Appl.}, 80(1-2):91--99, 1997.

\bibitem[Dyc74]{dyckhoff}
R.~Dyckhoff.
\newblock Perfect light maps as inverse limits.
\newblock {\em Quart J. Math. Oxford Ser. (2)}, 25:441--449, 1974.

\bibitem[Eil34]{eilenberg}
S.~Eilenberg.
\newblock Sur les transformations continues d'espaces métriques compacts.
\newblock {\em Fundamenta Mathematicae}, 1934.

\bibitem[Eng95]{engelking}
R.~Engelking.
\newblock {\em Theory of dimensions finite and infinite}, volume~10 of {\em
  Sigma Series in Pure Mathematics}.
\newblock Heldermann Verlag, Lemgo, 1995.

\bibitem[HM82]{heinmank}
S.~Heinrich and P.~Mankiewicz.
\newblock Applications of ultrapowers to the uniform and {L}ipschitz
  classification of {B}anach spaces.
\newblock {\em Studia Math.}, 73(3):225--251, 1982.

\bibitem[Kak39]{kakutani}
S.~Kakutani.
\newblock Mean ergodic theorem in abstract {$(L)$}-spaces.
\newblock {\em Proc. Imp. Acad., Tokyo}, 15:121--123, 1939.

\bibitem[Kir94]{kirchheim}
B.~Kirchheim.
\newblock Rectifiable metric spaces: local structure and regularity of the
  {H}ausdorff measure.
\newblock {\em Proc. Amer. Math. Soc.}, 121(1):113--123, 1994.

\bibitem[Laa00]{laakso}
T.~Laakso.
\newblock Ahlfors {$Q$}-regular spaces with arbitrary {$Q>1$} admitting weak
  {P}oincar\'e inequality.
\newblock {\em Geom. Funct. Anal.}, 10(1):111--123, 2000.

\bibitem[LP01]{LP01}
U.~Lang and C.~Plaut.
\newblock Bilipschitz embeddings of metric spaces into space forms.
\newblock {\em Geom. Dedicata}, 87(1-3):285--307, 2001.

\bibitem[Why34]{whyburn}
G.~T. Whyburn.
\newblock Non-{A}lternating {T}ransformations.
\newblock {\em Amer. J. Math.}, 56(1-4):294--302, 1934.

\end{thebibliography}
\bibliographystyle{alpha}
\end{document}